\documentclass[11pt,a4paper]{article}
\usepackage[margin=26mm]{geometry}
\usepackage[T1]{fontenc}
\usepackage[utf8]{inputenc}
\usepackage{lmodern}
\usepackage{amsmath,amssymb,amsthm,mathtools}
\usepackage{xurl}
\usepackage[unicode,hidelinks]{hyperref}
\usepackage{needspace}
\numberwithin{equation}{section}
\newtheorem{theorem}{Theorem}[section]
\newtheorem{lemma}[theorem]{Lemma}
\newtheorem{proposition}[theorem]{Proposition}
\newtheorem{corollary}[theorem]{Corollary}
\theoremstyle{definition}

\newtheorem{example}[theorem]{Example}
\newtheorem{remark}[theorem]{Remark}

\newcommand{\Ric}{\operatorname{Ric}^{M}}
\newcommand{\Scal}{\operatorname{Scal}^{M}}
\newcommand{\divg}{\operatorname{div}_g}
\newcommand{\gradN}{\operatorname{grad}^{N}}
\newcommand{\ddv}{\,dv_g}
\allowdisplaybreaks[2]
\title{Pohozaev identities and Liouville theorems\\for $L$-tension energies on manifolds}
\author{Xiangzhi Cao\thanks{aaa7756kijlp@163.com}\,\,\thanks{Nanjing Xiaozhuang University}}
\date{}
\begin{document}
\maketitle

\begin{abstract}
We study autonomous energies $B(\phi,e,\|\tau_L\|^2/2)$ for a general Lagrangian $L=L(y,r)$. The scaling derivative $\widehat L=2rL'$ determines the trace of the stress tensor and a vector-field identity on Riemannian domains. For homogeneous $L=a(y)(2r)^{p/2}/p$, the divergence term in this trace vanishes when $p=m$. A conformal vector field with a strictly signed conformal factor then gives $L$-harmonicity for positive mixed tension powers, under explicit integrability and flux conditions. This applies, in particular, to hyperbolic space. On complete gradient Ricci solitons with bounded Ricci curvature, a separate cutoff argument gives a global identity for general $L$. On the cigar and on Euclidean space, direct metric variations give second-order boundary criteria that also cover tension powers $1<q<2$. The applications include target weights and the fourth-power case. Examples distinguish the annular condition from global first-derivative integrability and show why scalar growth bounds or finite tension energy alone do not imply vanishing.
\end{abstract}

\section{Introduction and main results}\label{intro}
Let $\phi:(M^m,g)\to(N,h)$ be a smooth map and write $e=e(\phi)=\frac12\|d\phi\|^2$. For an autonomous Lagrangian $L=L(y,r)$, the $L$-energy and its tension field are
\begin{equation}\label{L-energy}
\begin{aligned}
E_L(\phi;D)&=\int_D L(\phi,e(\phi))\ddv,\\
\tau_L(\phi)&=\divg(L_\phi'd\phi)-(\gradN L)_\phi,\\
E_{2,L,q}(\phi;D)&=\frac1q\int_D\|\tau_L(\phi)\|^q\ddv.
\end{aligned}
\end{equation}

Here $q>1$, $D\Subset M$, and $L_\phi'=\partial_rL(\phi,e)$. Target derivatives are taken with $r$ fixed before evaluation. Criticality always means criticality under smooth compactly supported map variations on every such $D$. A map is $L$-harmonic if $\tau_L=0$.

The general results allow smooth autonomous $L(y,r)$. A principal specialization is the homogeneous Lagrangian
\begin{equation}\label{homogeneous-L}
L(y,r)=a(y)\frac{(2r)^{p/2}}p,\qquad a\in C^\infty(N),\quad a>0,\quad r\ge0.
\end{equation}

The associated autonomous energies are
\begin{equation}\label{B-functional}
E_{B,L}(\phi;D)=\int_D B(\phi,e,s)\ddv,
\qquad e=\tfrac12\|d\phi\|^2,\quad s=\tfrac12\|\tau_L\|^2.
\end{equation}

Primes on $B$ retain the notation $B_\phi'=\partial_rB$ and $B_\tau'=\partial_sB$, evaluated at $(\phi,e,s)$. The power density $B=(2s)^{q/2}/q$ gives $E_{2,L,q}$. If $a\equiv1$, then
\begin{equation}\label{energy}
\tau_p(\phi)=\divg(\|d\phi\|^{p-2}d\phi),\qquad
E_{p,q}(\phi;D)=\frac1q\int_D\|\tau_p(\phi)\|^q\ddv.
\end{equation}

The equations $\tau_L=0$ and $\tau_p=0$ need not be the ordinary harmonic map equation.

\paragraph{Conventions.}
The manifolds and metrics are smooth. General Lagrangians are smooth on a neighborhood of the relevant values in $N\times[0,\infty)$; the weighted power extensions below are handled separately. The auxiliary tension $\tau_F$ is defined by the same divergence formula for any such $F$, even when $F$ is not positive. The domain is connected, without boundary, and $m\ge2$. We write $\nabla$ for induced covariant derivatives, $\Delta=\divg\nabla$, and
$R^N(U,V)Z=\nabla_U\nabla_VZ-\nabla_V\nabla_UZ-\nabla_{[U,V]}Z$.
Repeated indices refer to a local orthonormal frame $\{e_i\}$. In $\nabla(L_\phi'd\phi)(e_i,e_j)$, the first argument is the differentiation direction.

For the homogeneous family, $p\ge2$. We allow zeros of $d\phi$ when $p=2$ or $p\ge4$; for $2<p<4$, we assume $d\phi\ne0$. These are sufficient regularity assumptions for the calculations below, not a claim about the optimal range at critical points. At allowed zeros, weighted products are differentiated as a whole and extended continuously. Terms carrying $p-2$ are omitted at $p=2$ before any negative power is evaluated.

For derivative estimates, put
\begin{equation}\label{B-u-v}
u=\|d\phi\|,\qquad v=\|\nabla d\phi\|,\qquad
G=u^{p-2}v+u^p,\qquad \Omega_R=B_{2R}(0)\setminus B_R(0).
\end{equation}

The quantity $G$ is used only for \eqref{homogeneous-L}. When $a$ is constant, $G=u^{p-2}v$ suffices; at $p=2$ the expression $u^{p-2}v$ means $v$.

\paragraph{Previous work and the scope of the results.}
The conservation-law approach begins with Baird and Eells \cite{r29} for harmonic maps. Jiang's bienergy variation \cite{r6} and the biharmonic stress tensor of Loubeau, Montaldo and Oniciuc \cite{r24} provide the corresponding higher-order framework. Dong and Wei \cite{r30} develop conservation laws and vanishing theorems for bundle-valued forms and first-order $F$-energies. Their results explain the role of domain geometry, but do not directly give the identities for $B(\phi,e,\|\tau_L\|^2/2)$. Contraction with a vector field, including a conformal vector field, is a standard step; see \cite[Lemmas 3.3--3.4]{r27}.

The functionals also have direct predecessors. Djaa and Mohammed Cherif \cite{r32} treat generalized $f$-bienergy, including weights depending on the target, and Cherif and Djaa \cite{r5} study the $L$-energy, its tension field and the associated bienergy. Mohammed Cherif \cite{r31} studies the squared $p$-tension functional. Mouffoki and Cherif \cite[Theorems 3.4--3.5 and Corollary 3.7]{r20} give its stress tensor, conservation law and trace. In particular, for $a=1,q=2$, the trace contains $(p-m)\divg\alpha$; its cancellation at $p=m$ is already present there. Their Corollary 3.7 assumes that the stress tensor itself vanishes. Below we use conservation together with a conformal vector field and a vanishing flux at infinity.

Latti and Mohammed Cherif \cite[Equation (1.5)]{r22} study exactly $E_{p,q}$ in \eqref{energy}, under the name ``$(p,q)$-harmonic maps''. Their Theorem 1 gives the first variation; their Theorems 11 and 12 give compact and complete noncompact Liouville results, respectively, for nonpositively curved targets. We use ``$(p,q)$-biharmonic'' only to distinguish these maps from other conventions; the functional is not introduced here. In particular, the exponent called $p$ in Han--Zhang \cite{r41} and Han--Luo \cite{r15} is the outer power of the ordinary tension field and corresponds to $q$ when our inner exponent is $2$. For $q=2$, we retain the original notation $E_{2,p}=E_{p,2}$, $S_{2,p}=S_{2,p,2}$ and $\tau_{2,p}=\tau_{2,p,2}$. Related curves and hypersurfaces are studied in \cite{r23}.

Liouville theorems on complete Riemannian domains already form a substantial literature. Baird, Fardoun and Ouakkas \cite{r33}, Nakauchi, Urakawa and Gudmundsson \cite{r34}, Maeta \cite{r35} and Luo \cite{r40,r38} give results for ordinary biharmonic maps under curvature and integrability assumptions. For example, \cite{r34} allows an arbitrary complete domain and a nonpositively curved target, with finite bienergy and either finite ordinary energy or infinite domain volume. Branding \cite{r39} and Branding--Luo \cite{r36} replace nonpositive target curvature by an upper curvature bound, with suitable domain hypotheses and smallness of the critical first-derivative norm. Usula \cite{r37} treats simple $b$-maps between conformally compact manifolds under boundary and target-curvature hypotheses, without imposing energy integrability. These results concern different hypotheses or classes of maps; passage from Euclidean space to a Riemannian domain alone is not a novelty claim.

For the applications below, Branding's soliton identity and cigar theorem \cite{r18} are direct precedents. In particular, Corollary 1.8 of the cited arXiv version already gives the classical cigar conclusion $a=1,q=2$. His Euclidean and interpolating sesqui-harmonic results \cite{r26,r27} are also relevant; the ordinary energy--bienergy case of Corollary \ref{B-first-order} belongs to this work. Minh and Tuyen \cite{r19} study $p$-harmonic maps on solitons.

Our purpose is to record the general-$L$ scaling field $\tau_{\widehat L}$, the resulting trace and boundary terms, and their consequences for weighted and mixed tension powers. The conformal vanishing theorem is a consequence of that identity and the known trace cancellation. It removes a target-curvature restriction at the cost of a conformal domain field and explicit flux control. Its hypotheses are not uniformly weaker than those in the preceding Liouville theorems.

\paragraph{Main results on Riemannian domains.}
For the following statements put $\widehat L=2rL'$, $W_B=B_\tau'\tau_L$, and
\[
\alpha_B=L_\phi'\langle d\phi,W_B\rangle,
\qquad
\Theta_B=(\widehat L_\phi'-mL_\phi')\langle d\phi,W_B\rangle.
\]
The tensor $S_{B,L}$ is the metric stress in \eqref{B-stress}. We require the composite density $(y,r,Z)\mapsto B(y,r,\|Z\|^2/2)$ to be $C^1$, $W_B$ to be $C^1$, and the weighted products in the stress and trace to be continuous. These requirements are explained in Section \ref{variation}. They include the homogeneous cases under the regularity conventions above and tension powers $q\ge2$.

\begin{theorem}[Vector-field identity]\label{general-multiplier}
Let $\phi:(M^m,g)\to(N,h)$ be a smooth critical point of $E_{B,L}$ for autonomous $L$ and $B$, with the stress regularity specified above. For a smooth vector field $X$, set
\begin{equation}\label{deformation-split}
\mathcal D_X=\tfrac12\mathcal L_Xg=\sigma g+K,
\qquad \sigma=\tfrac1m\divg X,\qquad \operatorname{tr}_gK=0.
\end{equation}
Write
\begin{equation}\label{pohozaev-density}
\mathcal P_{B,L}=mB-2eB_\phi'-\langle\tau_{\widehat L},W_B\rangle.
\end{equation}
If $X$ has compact support, then
\begin{equation}\label{general-multiplier-id}
\int_M\bigl[\sigma\mathcal P_{B,L}
-\Theta_B(\nabla\sigma)+\langle S_{B,L},K\rangle\bigr]\ddv=0.
\end{equation}
For general $X$, the same identity holds provided the three displayed terms are absolutely integrable and there are $\eta_j\in C_c^\infty(M)$ with $0\le\eta_j\le1$, $\eta_j\to1$ pointwise, such that
\begin{equation}\label{general-flux}
\lim_{j\to\infty}\int_M
\bigl(\|S_{B,L}\|\,\|X\|+|\sigma|\,\|\Theta_B\|\bigr)
\|\nabla\eta_j\|\ddv=0.
\end{equation}
\end{theorem}

The next statement applies on a complete Riemannian manifold carrying a conformal vector field. Completeness alone does not imply the existence of such a field or the required flux decay.

\begin{theorem}[Mixed tension powers on conformal domains]\label{conformal-main}
Let $(M^m,g)$ be complete and let $X$ be a smooth vector field satisfying $\mathcal L_Xg=2\sigma g$, where $\sigma$ is everywhere strictly positive or everywhere strictly negative. Let $L$ be given by \eqref{homogeneous-L} with $p=m$. Let $q_j\ge2$ and $b_j\in C^\infty(N)$, $b_j>0$, for $1\le j\le k$. Suppose $\phi:(M,g)\to(N,h)$ is a smooth critical point, with the regularity specified above, of the locally defined functional
\begin{equation}\label{conformal-mixed-energy}
E_{B,L}(\phi;D)=\sum_{j=1}^k\frac1{q_j}\int_D b_j(\phi)\|\tau_L\|^{q_j}\ddv.
\end{equation}
Assume
\begin{equation}\label{conformal-mixed-finite}
\int_M|\sigma|\sum_{j=1}^k b_j(\phi)\|\tau_L\|^{q_j}\ddv<\infty
\end{equation}
and that there are $\eta_\ell\in C_c^\infty(M)$ with $0\le\eta_\ell\le1$, $\eta_\ell\to1$ pointwise, and
\begin{equation}\label{conformal-mixed-flux}
\lim_{\ell\to\infty}\int_M\|S_{B,L}\|\,\|X\|\,\|\nabla\eta_\ell\|\ddv=0.
\end{equation}
Then $\tau_L=0$. No curvature assumption on $N$ is imposed.
\end{theorem}

Both proofs are in Section \ref{manifold-section}. The flux condition in Theorem \ref{conformal-main} involves the stress, which generally contains third derivatives of $\phi$; finite tension energy alone does not establish it. Corollary \ref{hyperbolic-cor} supplies an explicit sufficient weighted condition on $\mathbb H^m$, including $p=m=4$. Corollary \ref{cigar-main} and Theorem \ref{B-powers} give the cigar and Euclidean applications. Their direct metric-variation proofs use only first- and second-derivative bounds and also allow $1<q<2$.

For general $L$ on a Ricci soliton, Theorem \ref{general-B-soliton} gives an identity whose curvature terms require separate sign control. The Euclidean boundary criterion is isolated in Theorem \ref{general-euclidean}. Example \ref{stereo-example} satisfies its annular condition but has infinite ordinary energy. No optimality is claimed for the stated integrability conditions or for the mixed-power dimension restrictions.

For further background, see $F$-harmonic maps \cite{r1,r2}, weighted harmonic maps \cite{r3,r4}, $F$-biharmonic maps \cite{r7,r8}, $f$-biharmonic and bi-$f$-harmonic maps \cite{r9,r10,r11,r12}, and their foliated counterparts \cite{r13,r14}. The fourth-power soliton statements revise those in the author's earlier preprint \cite{r25}; the corresponding assertions there are replaced by the hypotheses and proofs given here.

Section \ref{variation} establishes the metric and scaling identities. Section \ref{manifold-section} treats general vector fields, Section \ref{global} treats Ricci solitons, and Section \ref{euclid} proves the Euclidean results. Section \ref{example} gives examples and obstructions. The map variation, expanded fourth-power formulas and explicit domain dependence are recorded in the appendices.

\section{Metric variation and scaling}\label{variation}\label{BL-results}
The metric variation below is valid for a general autonomous $L$ with the required differentiability. The general formulas use the radial scaling derivative; homogeneity is imposed only in the indicated specializations.

\begin{lemma}[Metric linearization]\label{metric-lemma}
Let $k=\delta g$ be a compactly supported symmetric metric variation and set
$\xi_k=(\divg k)^\sharp-\frac12\nabla\operatorname{tr}_gk$. Then
\begin{equation}\label{linearize-metric}
\begin{aligned}
\delta_g\tau_L={}&-\frac12\divg\bigl(L_\phi''\langle\phi^*h,k\rangle d\phi\bigr)
-k_{ij}\nabla(L_\phi'd\phi)(e_i,e_j)\\
&-L_\phi'd\phi(\xi_k)
+\frac12\langle\phi^*h,k\rangle(\gradN L')_\phi.
\end{aligned}
\end{equation}

\end{lemma}
\begin{proof}
One has $\delta_ge=-\frac12\langle\phi^*h,k\rangle$. With the scalar coefficient $L_\phi'$ held fixed, varying the metric trace and the domain connection in its divergence gives
$-k_{ij}\nabla(L_\phi'd\phi)(e_i,e_j)-L_\phi'd\phi(\xi_k)$.
The variation of the coefficient is $\delta_gL_\phi'=-\frac12L_\phi''\langle\phi^*h,k\rangle$.
Finally, varying $-(\gradN L)_\phi$ contributes
$\frac12\langle\phi^*h,k\rangle(\gradN L')_\phi$.
Adding these terms proves the formula, including the sign of its last term.
\end{proof}

\begin{lemma}[Conformal metric variation]\label{conformal-linearization}
Define the scaling derivative
\begin{equation}\label{Lhat-definition}
\widehat L(y,r)=2rL'(y,r).
\end{equation}
For a scalar function $\sigma$, the metric linearization in direction $2\sigma g$ is
\begin{equation}\label{conformal-linearization-eq}
\delta_g\tau_L[2\sigma g]
=-\sigma\tau_{\widehat L}
-\bigl[2eL_\phi''-(m-2)L_\phi'\bigr]d\phi(\nabla\sigma).
\end{equation}
In particular $\delta_g\tau_L[2g]=-\tau_{\widehat L}$. If $L$ is given by \eqref{homogeneous-L}, then $\widehat L=pL$ and $\tau_{\widehat L}=p\tau_L$.
\end{lemma}
\begin{proof}
For $k=2\sigma g$, one has $\langle\phi^*h,k\rangle=4\sigma e$ and
$\xi_k=(2-m)\nabla\sigma$. Substitution in \eqref{linearize-metric} gives the terms without $\nabla\sigma$ as
\[
-2\sigma\divg\bigl((L_\phi'+eL_\phi'')d\phi\bigr)
+2\sigma e(\gradN L')_\phi=-\sigma\tau_{\widehat L}.
\]
The remaining terms give \eqref{conformal-linearization-eq}. This is a metric linearization statement, not stationarity under arbitrary conformal changes of $g$.
\end{proof}

We require the composite density $(y,r,Z)\mapsto B(y,r,\|Z\|_h^2/2)$, $Z\in T_yN$, to be $C^1$. Products involving $B_\tau'$ are interpreted through this composite at $Z=0$. This includes $B=(2s)^{q/2}/q$ for all $q>1$. Define
\begin{equation}\label{B-aux}
\begin{aligned}
W_B&=B_\tau'\tau_L,\qquad
\alpha_B(X)=L_\phi'\langle d\phi(X),W_B\rangle,\\
\Psi_B&=L_\phi''\langle d\phi,\nabla W_B\rangle
+\langle(\gradN L')_\phi,W_B\rangle.
\end{aligned}
\end{equation}

For stress identities, we additionally require $W_B\in C^1$ and continuity of the displayed weighted products. The unintegrated metric formula and the Euclidean argument do not require derivatives of $W_B$.

We also use the one-forms
\begin{equation}\label{hat-alpha-theta}
\begin{aligned}
\widehat\alpha_B(X)&=\widehat L_\phi'\langle d\phi(X),W_B\rangle,\\
\Theta_B&=\widehat\alpha_B-m\alpha_B
=\bigl[2eL_\phi''-(m-2)L_\phi'\bigr]\langle d\phi,W_B\rangle.
\end{aligned}
\end{equation}
The field $\tau_{\widehat L}$ is evaluated along the same map; no criticality for a second functional is assumed.

\begin{lemma}\label{B-trace-lemma}
The metric variation satisfies
\begin{equation}\label{B-raw}
\delta_gE_{B,L}
=\int_M\left[-\tfrac12B_\phi'\langle\phi^*h,k\rangle
+\langle W_B,\delta_g\tau_L\rangle
+\tfrac12B\operatorname{tr}_gk\right]\ddv
=\tfrac12\int_M\langle S_{B,L},k\rangle\ddv,
\end{equation}
where
\begin{equation}\label{B-stress}
\begin{aligned}
S_{B,L}(X,Y)={}&(B-\divg\alpha_B)g(X,Y)\\
&+L_\phi'\bigl(\langle d\phi(X),\nabla_YW_B\rangle
+\langle d\phi(Y),\nabla_XW_B\rangle\bigr)\\
&+(\Psi_B-B_\phi')\langle d\phi(X),d\phi(Y)\rangle.
\end{aligned}
\end{equation}
At a critical map, $\divg S_{B,L}=0$ distributionally. For general autonomous $L$,
\begin{equation}\label{general-trace}
\operatorname{tr}_gS_{B,L}
=mB-2eB_\phi'-\langle\tau_{\widehat L},W_B\rangle
+\divg\Theta_B.
\end{equation}
For \eqref{homogeneous-L}, this specializes to
\begin{equation}\label{B-trace}
\operatorname{tr}_g S_{B,L}
=mB-2eB_\phi'-2psB_\tau'+(p-m)\divg\alpha_B.
\end{equation}
\end{lemma}

\begin{proof}
The first equality of \eqref{B-raw} is the chain rule. To derive the stress, substitute \eqref{linearize-metric}. Integrating the $\xi_k$ term gives
$\nabla_i(\alpha_B)_j+\nabla_j(\alpha_B)_i-(\divg\alpha_B)g_{ij}$.
The terms containing $\nabla(L_\phi'd\phi)$ cancel those already present in the metric linearization. The remaining coefficients of $\phi^*h$ are $\Psi_B-B_\phi'$, giving \eqref{B-stress}.

For a compactly supported vector field $X$, simultaneous pullback of the metric and map preserves the autonomous functional. At a smooth critical map this yields
$\delta_gE_{B,L}(\mathcal L_Xg)=0$, and hence $\divg S_{B,L}=0$ distributionally. We use only metric variations of this form for stationarity, not stationarity under arbitrary symmetric tensors.

For the general trace, set $Q=\langle d\phi,\nabla W_B\rangle$. Tracing \eqref{B-stress} gives
\[
mB-2eB_\phi'-m\divg\alpha_B
+2(L_\phi'+eL_\phi'')Q
+2e\langle(\gradN L')_\phi,W_B\rangle.
\]
The last two terms equal $\divg\widehat\alpha_B-\langle\tau_{\widehat L},W_B\rangle$. This proves \eqref{general-trace}.

For \eqref{homogeneous-L},
\begin{equation}\label{homogeneity}
2rL''=(p-2)L',\qquad 2r\gradN L'=p\gradN L.
\end{equation}

Also,
\[
\divg\alpha_B=2sB_\tau'+\langle(\gradN L)_\phi,W_B\rangle
+L_\phi'\langle d\phi,\nabla W_B\rangle.
\]
Tracing \eqref{B-stress} and applying \eqref{homogeneity} gives \eqref{B-trace}.
\end{proof}

For $B=(2s)^{q/2}/q$ and $q\ge2$, we write
\begin{equation}\label{W}
W=\|\tau_L(\phi)\|^{q-2}\tau_L(\phi),\qquad W=0\text{ at }\tau_L=0,
\end{equation}
\begin{equation}\label{Psi}
\Psi_{L,q}=L_\phi''\langle d\phi,\nabla W\rangle
+\langle(\gradN L')_\phi,W\rangle.
\end{equation}
\begin{equation}\label{alpha}
\alpha(X)=L_\phi'\langle d\phi(X),W\rangle.
\end{equation}

We denote the stress by $S_{2,L,q}$, and use $\widehat\alpha$ and $\Theta$ for the corresponding power one-forms. In this case
\begin{equation}\label{alpha-div}
\divg\alpha=\|\tau_L\|^q
+\langle(\gradN L)_\phi,W\rangle
+L_\phi'\langle d\phi,\nabla W\rangle.
\end{equation}
For the homogeneous family, the trace is
\begin{equation}\label{trace}
\operatorname{tr}_gS_{2,L,q}
=\left(\frac mq-p\right)\|\tau_L\|^q+(p-m)\divg\alpha.
\end{equation}

For $a=1$ and $q=2$, \eqref{trace} is the trace identity in \cite[Corollary 3.7]{r20}; in particular, its divergence term already vanishes at $p=m$. The explicit power stress and map variation are in Appendix \ref{map-appendix}.

\begin{lemma}[Second-order bounds]\label{basic-bounds}
For \eqref{homogeneous-L}, suppose $a$ and $\|\gradN a\|$ are bounded on the image. With $u,v,G$ as above,
\begin{equation}\label{basic-bounds-eq}
\|\tau_L\|+\|\nabla(L_\phi'd\phi)\|\le CG,
\qquad |L_\phi'|u\le Cu^{p-1}.
\end{equation}
If $a$ is constant, the term $u^p$ is omitted from $G$.
\end{lemma}
\begin{proof}
The identities
\[
L_\phi'=a(\phi)u^{p-2},\qquad
\tau_L=a(\phi)\tau_p+u^{p-2}d\phi(\nabla(a\circ\phi))
-\frac{u^p}{p}(\gradN a)\circ\phi
\]
give the tension estimate. The product rule gives the same bound for $\nabla(L_\phi'd\phi)$. Derivatives of $a\circ\phi$ contribute $Cu^p$; all other contributions are bounded by $Cu^{p-2}v$. No target Hessian bound on $a$ is used.
\end{proof}

\begin{remark}[Conformal interpretation of the target weight]\label{conformal-weight}
For \eqref{homogeneous-L}, set $\widetilde h=a^{2/p}h$. Then $E_L$ is the ordinary $p$-energy computed with target metric $\widetilde h$. Comparing its first variations in the two target metrics gives
\begin{equation}\label{conformal-tension}
\tau_L^{h}=a(\phi)^{2/p}\tau_p^{\widetilde h},
\qquad
\|\tau_L^{h}\|_h^q=a(\phi)^{q/p}\|\tau_p^{\widetilde h}\|_{\widetilde h}^q.
\end{equation}
Indeed, $h(\tau_L,V)=\widetilde h(\tau_p^{\widetilde h},V)$ for every compactly supported variation field $V$. Thus the target-dependent power energy can also be viewed as a weighted $p$-tension energy after a conformal target change. We retain $h$ and $a$ because the derivative assumptions in the theorems are measured in the original metric. No uniform equivalence between $h$ and $\widetilde h$ is assumed.
\end{remark}

\section{Pohozaev identities on general manifolds}\label{manifold-section}
We prove the main theorems using the trace formula of Lemma \ref{B-trace-lemma}. The contraction of a conserved symmetric tensor with a vector field is standard; see \cite[Lemmas 3.3--3.4]{r27}. The terms specific to the present functional are $\mathcal P_{B,L}$ and $\Theta_B$. The stress is continuous and its divergence vanishes distributionally.

\begin{proof}[Proof of Theorem \ref{general-multiplier}]
Testing conservation against $\eta_jX$ gives
\[
0=\int_M\eta_j\langle S_{B,L},\mathcal D_X\rangle\ddv
+\int_MS_{B,L}(X,\nabla\eta_j)\ddv.
\]
Use \eqref{general-trace} and integrate $\eta_j\sigma\divg\Theta_B$ by parts. The exact localized formula is
\[
\begin{aligned}
0={}&\int_M\eta_j
\bigl[\sigma\mathcal P_{B,L}-\Theta_B(\nabla\sigma)
+\langle S_{B,L},K\rangle\bigr]\ddv\\
&+\int_M\bigl[S_{B,L}(X,\nabla\eta_j)
-\sigma\Theta_B(\nabla\eta_j)\bigr]\ddv.
\end{aligned}
\]
The second line tends to zero by \eqref{general-flux}; dominated convergence handles the first line. For compactly supported $X$, take a cutoff equal to one on its support. No curvature assumption is used.
\end{proof}

Theorem \ref{general-multiplier} separates three effects: scaling of the density, variation of the conformal factor, and the trace-free part of the deformation tensor. A positive Hessian of a function does not by itself give the required sign, since the higher-order stress need not be positive or negative semidefinite.

\begin{corollary}[Conformal vector fields]\label{conformal-cor}
Under the integrability and flux assumptions of Theorem \ref{general-multiplier}, suppose $\mathcal L_Xg=2\sigma g$. Then
\begin{equation}\label{conformal-global}
\int_M\bigl[\sigma\mathcal P_{B,L}-\Theta_B(\nabla\sigma)\bigr]\ddv=0.
\end{equation}
For \eqref{homogeneous-L} and $B=\|\tau_L\|^q/q$, $q\ge2$, this reads
\begin{equation}\label{conformal-power}
\int_M\left[\sigma\left(\frac mq-p\right)\|\tau_L\|^q
-(p-m)\alpha(\nabla\sigma)\right]\ddv=0.
\end{equation}
If $p=m$ and $\sigma$ is everywhere strictly positive or everywhere strictly negative, then $\tau_L=0$.
\end{corollary}
\begin{proof}
Here $K=0$. In the homogeneous case, $\tau_{\widehat L}=p\tau_L$ and $\Theta_B=(p-m)\alpha$. If $p=m$, the second term vanishes and $m/q-m<0$, so the remaining integrand has one sign and vanishes only when $\tau_L=0$.
\end{proof}

For a rotationally symmetric metric $g=dr^2+\omega(r)^2g_{S^{m-1}}$ which extends smoothly over a pole, $X=\omega(r)\partial_r$ satisfies $\mathcal L_Xg=2\omega'(r)g$. Under the hypotheses above, \eqref{conformal-power} becomes
\[
\int_M\left[
\left(\frac mq-p\right)\omega'\|\tau_L\|^q
-(p-m)\omega''\alpha(\partial_r)\right]\ddv=0.
\]
The choices $\omega=r$ and $\omega=\sinh r$ give Euclidean and hyperbolic space, respectively. In the latter case the flux and integrability assumptions carry the corresponding growth weights. The term involving $\omega''$ cannot be discarded. The next corollary gives the mixed-power form, followed by the proof of Theorem \ref{conformal-main} and a hyperbolic application. The cigar calculation in Section \ref{cigar-section} avoids derivatives of $W$ and has directly stated second-order assumptions.

\begin{corollary}[Mixed-power identity for general $p$]\label{general-mixed-cor}
Let $L$ be given by \eqref{homogeneous-L} and let
\[
B(y,r,s)=\sum_{j=1}^k b_j(y)\frac{(2s)^{q_j/2}}{q_j},
\qquad b_j>0,\quad q_j\ge2.
\]
Under the hypotheses of Theorem \ref{general-multiplier},
\begin{equation}\label{general-mixed-identity}
\int_M\left[\sigma\sum_{j=1}^k\left(\frac m{q_j}-p\right)b_j(\phi)\|\tau_L\|^{q_j}
-(p-m)\alpha_B(\nabla\sigma)+\langle S_{B,L},K\rangle\right]\ddv=0.
\end{equation}
Here $W_B=\sum_j b_j(\phi)\|\tau_L\|^{q_j-2}\tau_L$. The integrability and flux hypotheses apply to this density and stress.
\end{corollary}
\begin{proof}
One has $B_\phi'=0$, $\tau_{\widehat L}=p\tau_L$ and $\Theta_B=(p-m)\alpha_B$. Therefore
\[
\mathcal P_{B,L}=\sum_{j=1}^k\left(\frac m{q_j}-p\right)b_j(\phi)\|\tau_L\|^{q_j}.
\]
Substitution in \eqref{general-multiplier-id} proves the identity. The target derivatives of $b_j$ are already included in the stress through $\nabla W_B$; there is no additional explicit domain source.
\end{proof}

\begin{proof}[Proof of Theorem \ref{conformal-main}]
Since $p=m$ and $X$ is conformal, $\Theta_B=0$ and $K=0$. Condition \eqref{conformal-mixed-finite} gives absolute integrability of $\sigma\mathcal P_{B,L}$, and \eqref{conformal-mixed-flux} is precisely the remaining flux condition in Theorem \ref{general-multiplier}. Formula \eqref{general-mixed-identity} becomes
\begin{equation}\label{conformal-mixed-vanishing}
0=-m\int_M\sigma\sum_{j=1}^k\frac{q_j-1}{q_j}b_j(\phi)\|\tau_L\|^{q_j}\ddv.
\end{equation}
The summands are nonnegative, their coefficients are strictly positive, and $\sigma$ has a strict sign. Hence $\tau_L=0$ almost everywhere and, by continuity, everywhere.
\end{proof}

\begin{corollary}[Hyperbolic domains and the fourth-power case]\label{hyperbolic-cor}
Let $M=\mathbb H^m$ with sectional curvature $-1$, let $p=m$, and assume the criticality and regularity hypotheses of Theorem \ref{conformal-main}. Fix $o\in\mathbb H^m$ and put $r=d_g(o,\cdot)$. If
\begin{equation}\label{hyperbolic-weighted}
\int_{\mathbb H^m}\cosh r\left(\sum_{j=1}^k b_j(\phi)\|\tau_L\|^{q_j}+\|S_{B,L}\|\right)\ddv<\infty,
\end{equation}
then $\tau_L=0$. In particular, for $m=p=4$ and $a=1$, this gives $\tau_4=0$ for critical points of positive mixed powers of $\tau_4$ satisfying \eqref{hyperbolic-weighted}.
\end{corollary}
\begin{proof}
The smooth vector field $X=\nabla(\cosh r)=\sinh r\,\partial_r$ satisfies $\mathcal L_Xg=2\cosh r\,g$. Choose a radial cutoff $\eta_R$ equal to one on $B_R(o)$ and zero outside $B_{2R}(o)$, with $\|\nabla\eta_R\|\le C/R$. It is smooth since it is constant near the pole. Then
\[
\int_{\mathbb H^m}\|S_{B,L}\|\,\|X\|\,\|\nabla\eta_R\|\ddv
\le\frac CR\int_{B_{2R}\setminus B_R}\cosh r\,\|S_{B,L}\|\ddv\longrightarrow0.
\]
The other integrability condition follows directly from \eqref{hyperbolic-weighted}. Apply Theorem \ref{conformal-main}.
\end{proof}

Condition \eqref{hyperbolic-weighted} is sufficient and may be strong: it includes both the growth of the conformal field and a third-derivative quantity. It is not implied by finite $E_{B,L}$, and is not claimed to improve the curvature-based criteria in \cite{r34,r35,r38,r36}. For $p\ne m$, the $\alpha_B(\nabla\sigma)$ term in \eqref{general-mixed-identity} remains and needs separate control.

\section{General Lagrangians on gradient Ricci solitons}\label{global}
We normalize the soliton equation by \cite{r16,r17}
\begin{equation}\label{soliton}
\Ric+\nabla^2f=\lambda g,\qquad \lambda\in\mathbb R.
\end{equation}
\begin{equation}\label{soliton-basic}
\Delta f=m\lambda-\Scal,\qquad \nabla\Scal=2\Ric(\nabla f).
\end{equation}

The calculation below follows the stress-energy method of \cite{r18}. Its principal output is an integral identity; vanishing requires a sign assumption on the remaining curvature contractions. In the ordinary biharmonic case, \cite[Theorem 1.5 and Remark 1.9]{r18} uses an upper Ricci bound together with a bound on the potential gradient, or an additional weighted boundary condition. Here the two-sided Ricci bound gives linear growth of $\nabla f$, which is combined with the annular tail estimates below. These hypotheses are not asserted to be uniformly weaker than those in \cite{r18}.

\begin{theorem}[The general $B$--$L$ soliton identity]\label{general-B-soliton}
Let $(M,g,f)$ be a complete gradient Ricci soliton with bounded Ricci curvature, and let $L=L(y,r)$ be autonomous. Let $\phi$ be critical for $E_{B,L}$ with the regularity of Lemma \ref{B-trace-lemma}. Write $H=\nabla(L_\phi'd\phi)$ and $\beta=\Psi_B-B_\phi'$. Suppose
\begin{equation}\label{general-B-integrability}
\int_M\left[|B|+e|B_\phi'|+\|\tau_{\widehat L}\|\,\|W_B\|
+\|\alpha_B\|+\|\widehat\alpha_B\|
+\|H\|\,\|W_B\|+|\beta|\,\|d\phi\|^2\right]\ddv<\infty.
\end{equation}
Then
\begin{equation}\label{general-B-soliton-id}
\begin{aligned}
0={}&\int_M\left[(m\lambda-\Scal)B-2\lambda eB_\phi'
-\lambda\langle\tau_{\widehat L},W_B\rangle\right]\ddv\\
&+2\int_M\Ric_{ij}\langle H_{ij},W_B\rangle\ddv
-\int_M\beta\Ric_{ij}(\phi^*h)_{ij}\ddv.
\end{aligned}
\end{equation}
\end{theorem}
\begin{proof}
For this proof write $H=\nabla(L_\phi'd\phi)$ and $\beta=\Psi_B-B_\phi'$. The stress has the form
\begin{equation}\label{B-stress-alpha}
(S_{B,L})_{ij}=(B-\divg\alpha_B)g_{ij}
+\nabla_i(\alpha_B)_j+\nabla_j(\alpha_B)_i
-\langle H_{ij}+H_{ji},W_B\rangle+\beta(\phi^*h)_{ij}.
\end{equation}
Test conservation against $\psi\nabla f$, for $\psi\in C_c^\infty(M)$:
\[
\int_M\psi\langle\lambda g-\Ric,S_{B,L}\rangle\ddv
=-\int_M S_{B,L}(\nabla f,\nabla\psi)\ddv.
\]
In the Ricci contraction of \eqref{B-stress-alpha}, integration by parts gives
\begin{equation}\label{B-ric-cut}
\begin{aligned}
\int_M\psi\langle\Ric,S_{B,L}\rangle\ddv
={}&\int_M\psi\Scal B\ddv-2\int_M\psi\Ric_{ij}\langle H_{ij},W_B\rangle\ddv\\
&+\int_M\psi\beta\Ric_{ij}(\phi^*h)_{ij}\ddv\\
&\quad
+\int_M\Scal\alpha_B(\nabla\psi)\ddv
-2\int_M\alpha_B(\Ric(\nabla\psi))\ddv.
\end{aligned}
\end{equation}
The two terms involving $d\Scal$ cancel by $\divg\Ric=\frac12d\Scal$. In the boundary contraction of the same stress formula, integration by parts cancels all Hessians of $\psi$ and leaves
\begin{equation}\label{B-stress-boundary}
\begin{aligned}
\int_M S_{B,L}(\nabla f,\nabla\psi)\ddv
={}&\int_M B\langle\nabla f,\nabla\psi\rangle\ddv
-\int_M(\Delta\psi)\alpha_B(\nabla f)\ddv\\
&-\int_M(\Delta f)\alpha_B(\nabla\psi)\ddv\\
&-\int_M\langle H(\nabla f,\nabla\psi)+H(\nabla\psi,\nabla f),W_B\rangle\ddv\\
&+\int_M\beta\langle d\phi(\nabla f),d\phi(\nabla\psi)\rangle\ddv.
\end{aligned}
\end{equation}
For example, $-(\divg\alpha_B)\langle\nabla f,\nabla\psi\rangle$ integrates to $\alpha_B(\nabla\langle\nabla f,\nabla\psi\rangle)$; this cancels the Hessian contributions from the two derivatives of $\alpha_B$. Combining \eqref{general-trace}, \eqref{B-ric-cut} and \eqref{B-stress-boundary}, the left side of \eqref{general-B-soliton-id} with a factor $\psi$ in every integral equals
\begin{equation}\label{B-boundary}
\begin{aligned}
&-2\int_M\alpha_B(\Ric(\nabla\psi))\ddv
+\lambda\int_M\widehat\alpha_B(\nabla\psi)\ddv\\
&-\int_M B\langle\nabla f,\nabla\psi\rangle\ddv
+\int_M(\Delta\psi)\alpha_B(\nabla f)\ddv\\
&+\int_M\langle H(\nabla f,\nabla\psi)
+H(\nabla\psi,\nabla f),W_B\rangle\ddv\\
&-\int_M\beta\langle d\phi(\nabla f),d\phi(\nabla\psi)\rangle\ddv.
\end{aligned}
\end{equation}

The one-form in the second term is $\widehat\alpha_B$, because $\Theta_B+m\alpha_B=\widehat\alpha_B$ and $\Delta f=m\lambda-\Scal$. For homogeneous $L$, this term is $p\lambda\int\alpha_B(\nabla\psi)$.

If $M$ is compact, take $\psi=1$. Otherwise bounded Ricci curvature implies $\|\nabla^2f\|\le C$, and integration along minimizing geodesics gives
\begin{equation}\label{growth}
\|\nabla f\|(x)\le C(1+d_g(x,x_0)),\qquad |\Delta f|\le C.
\end{equation}

Corollary 2.3 of \cite{r21}, with parameter $\alpha=0$, supplies cutoffs
\begin{equation}\label{cutoffs}
\begin{gathered}
0\le\eta_R\le1,\quad \eta_R=1\text{ on }B_R(x_0),\quad
\operatorname{supp}\eta_R\subset B_{\gamma R}(x_0),\\
\|\nabla\eta_R\|+|\Delta\eta_R|\le C/R,\qquad R\ge1.
\end{gathered}
\end{equation}

Set $\psi_R=\eta_R^2$ and, in this proof, $\Omega_R=B_{\gamma R}(x_0)\setminus B_R(x_0)$. Formula \eqref{B-boundary} is bounded in absolute value by
\[
\frac CR\int_{\Omega_R}(\|\alpha_B\|+\|\widehat\alpha_B\|)\ddv
+\frac{C(1+R)}R\int_{\Omega_R}
\bigl(|B|+\|\alpha_B\|+\|H\|\,\|W_B\|+|\beta|\,\|d\phi\|^2\bigr)\ddv.
\]
This tends to zero by \eqref{general-B-integrability}. The interior integrals converge absolutely by the same hypothesis and bounded Ricci curvature. No uniform bound on $\nabla f$ or integrability of the full stress tensor is imposed.
\end{proof}
\begin{corollary}[General tension powers on solitons]\label{general-power-soliton}
Let $q\ge2$ and let $\phi$ be a smooth critical point of $E_{2,L,q}$ for general autonomous $L$ on a complete gradient Ricci soliton with bounded Ricci curvature. Suppose
\begin{equation}\label{general-power-soliton-I}
\begin{aligned}
\int_M\Bigl[&\|\tau_L\|^q
+\|\tau_{\widehat L}\|\,\|W\|+\|\alpha\|+\|\widehat\alpha\|\\
&+\|\nabla(L_\phi'd\phi)\|\,\|W\|
+|\Psi_{L,q}|\,\|d\phi\|^2\Bigr]\ddv<\infty.
\end{aligned}
\end{equation}
Then
\begin{equation}\label{general-power-soliton-id}
\begin{aligned}
0={}&\int_M\bigl[(m\lambda-\Scal)\|\tau_L\|^q
-q\lambda\langle\tau_{\widehat L},W\rangle\bigr]\ddv\\
&+2q\int_M\Ric_{ij}\langle\nabla(L_\phi'd\phi)(e_i,e_j),W\rangle\ddv\\
&-q\int_M\Psi_{L,q}\Ric_{ij}(\phi^*h)_{ij}\ddv.
\end{aligned}
\end{equation}
\end{corollary}
\begin{proof}
Set $B=(2s)^{q/2}/q$ in Theorem \ref{general-B-soliton} and multiply the identity by $q$.
\end{proof}

\begin{corollary}[A quantitative nonhomogeneous criterion]\label{soliton-defect-cor}
Assume the hypotheses of Corollary \ref{general-power-soliton}. Fix $p\in\mathbb R$, let $D_p=2rL'-pL$, and suppose
\begin{equation}\label{soliton-defect-assumptions}
\|\tau_{D_p}\|\le\varepsilon\|\tau_L\|,
\qquad
\Scal<\lambda(m-pq)-q|\lambda|\varepsilon
\end{equation}
everywhere, for some $\varepsilon\ge0$. If the sum of the last two integrals in \eqref{general-power-soliton-id} is nonnegative, then $\tau_L=0$.
\end{corollary}
\begin{proof}
Linearity gives $\tau_{\widehat L}=p\tau_L+\tau_{D_p}$. The first integrand in \eqref{general-power-soliton-id} is bounded below by
\[
\bigl[\lambda(m-pq)-\Scal-q|\lambda|\varepsilon\bigr]\|\tau_L\|^q.
\]
The bracket is strictly positive, and the remaining integral is nonnegative. All terms must vanish. This also covers $\lambda=0$, with the scalar-curvature condition interpreted as written.
\end{proof}
The hypotheses compare tension fields along a particular map. They are not consequences of scalar convexity or growth bounds for $L$; see Example \ref{jet-example}. The curvature-integral sign is also an additional assumption. The identity, rather than an unconditional vanishing claim, is the general soliton statement.

\begin{corollary}[Homogeneous Lagrangians]\label{B-soliton}
Let $(M,g,f)$ be a complete gradient Ricci soliton with bounded Ricci curvature. Let $L$ be given by \eqref{homogeneous-L}. Let $\phi$ be critical for \eqref{B-functional}, with the regularity of Lemma \ref{B-trace-lemma}. Suppose
\begin{equation}\label{B-integrability}
\begin{aligned}
\int_M\Bigl(&|B|+e|B_\phi'|+s|B_\tau'|+\|\alpha_B\|\\
&+\|\nabla(L_\phi'd\phi)\|\,\|W_B\|
+|\Psi_B-B_\phi'|\,\|d\phi\|^2\Bigr)\ddv<\infty.
\end{aligned}
\end{equation}
Then
\begin{equation}\label{B-soliton-id}
\begin{aligned}
0={}&\int_M\bigl[(m\lambda-\Scal)B
-2\lambda eB_\phi'-2p\lambda sB_\tau'\bigr]\ddv\\
&+2\int_M\Ric(e_i,e_j)
\langle\nabla(L_\phi'd\phi)(e_i,e_j),W_B\rangle\ddv\\
&-\int_M(\Psi_B-B_\phi')\Ric(e_i,e_j)
\langle d\phi(e_i),d\phi(e_j)\rangle\ddv.
\end{aligned}
\end{equation}
\end{corollary}

\begin{proof}
For this $L$, $\tau_{\widehat L}=p\tau_L$ and $\widehat\alpha_B=p\alpha_B$. Hence \eqref{B-integrability} implies \eqref{general-B-integrability}, and Theorem \ref{general-B-soliton} gives \eqref{B-soliton-id}.
\end{proof}

\begin{corollary}[Weighted homogeneous densities]\label{B-hom-cor}
Under the assumptions of Corollary \ref{B-soliton}, suppose $B\ge0$ and, along the map, for some $d\in\mathbb R$,
\begin{equation}\label{B-degree}
2eB_\phi'+2psB_\tau'=dB.
\end{equation}
Assume also that $B(\phi,e,s)=0$ implies $s=0$. If $\Scal<\lambda(m-d)$ everywhere and the sum of the last two integrals in \eqref{B-soliton-id} is nonnegative, then $\tau_L=0$.

If $m=2$ and $\Scal$ is constant, then instead
\begin{equation}\label{B-surface}
(\lambda-\tfrac12\Scal)(2-d)\int_M B\ddv=0.
\end{equation}
Thus $\Scal\ne2\lambda$ and $d\ne2$ imply $\tau_L=0$, without the curvature-integral sign condition.
\end{corollary}
\begin{proof}
Substitute \eqref{B-degree} into \eqref{B-soliton-id}. The first integrand is nonnegative and has a strictly positive coefficient of $B$; the remaining sum is nonnegative by hypothesis. Thus $B=0$, and the zero-set condition gives $\tau_L=0$. On a surface, test conservation against $\psi_R\nabla f$ and use $\Ric=(\Scal/2)g$ and \eqref{B-trace}. The boundary estimates above give \eqref{B-surface}.
\end{proof}
For $B=(2s)^{q/2}/q$, the scaling degree is $d=pq$. Condition \eqref{B-degree} also allows target-dependent factors and densities involving both $r$ and $s$. The sign condition on the curvature contractions remains an additional assumption.

\begin{corollary}[The power identity]\label{global-lemma}
Let $q\ge2$ and let $(M,g,f)$ be a complete gradient Ricci soliton with bounded Ricci curvature. Let $L$ be given by \eqref{homogeneous-L} and let $\phi$ be critical for $E_{2,L,q}$, with the conventions of Section \ref{intro}. Assume
\begin{equation}\label{IL}
\begin{aligned}
\int_M\Bigl[&(|L_\phi'|\,\|d\phi\|)^q+\|\tau_L\|^q
+\|\nabla(L_\phi'd\phi)\|\,\|\tau_L\|^{q-1}\\
&+|\Psi_{L,q}|\,\|d\phi\|^2\Bigr]\ddv<\infty.
\end{aligned}
\end{equation}

Then
\begin{equation}\label{global-id}
\begin{aligned}
0={}&\int_M[\lambda(m-pq)-\Scal]\|\tau_L\|^q\ddv\\
&+2q\int_M\Ric(e_i,e_j)
\langle\nabla(L_\phi'd\phi)(e_i,e_j),W\rangle\ddv\\
&-q\int_M\Psi_{L,q}\Ric(e_i,e_j)
\langle d\phi(e_i),d\phi(e_j)\rangle\ddv.
\end{aligned}
\end{equation}

\end{corollary}
\begin{proof}
For $B=(2s)^{q/2}/q$, $W_B=W$, $\Psi_B=\Psi_{L,q}$, and
\begin{equation}\label{young}
\|\alpha\|\le |L_\phi'|\,\|d\phi\|\,\|\tau_L\|^{q-1}
\le\frac1q(|L_\phi'|\,\|d\phi\|)^q
+\frac{q-1}{q}\|\tau_L\|^q.
\end{equation}

Thus \eqref{IL} implies \eqref{B-integrability}. Multiply \eqref{B-soliton-id} by $q$.
\end{proof}

\begin{corollary}[A conditional vanishing criterion]\label{main-thm}
Under all assumptions of Corollary \ref{global-lemma}, suppose $\Scal<\lambda(m-pq)$ everywhere and
\begin{equation}\label{sign}
\begin{aligned}
&2\int_M\Ric(e_i,e_j)
\langle\nabla(L_\phi'd\phi)(e_i,e_j),W\rangle\ddv\\
&\qquad-\int_M\Psi_{L,q}\Ric(e_i,e_j)
\langle d\phi(e_i),d\phi(e_j)\rangle\ddv\ge0.
\end{aligned}
\end{equation}
Then $\tau_L(\phi)=0$.
\end{corollary}
\begin{proof}
In \eqref{global-id}, the first integral is nonnegative and the sum of the other two is nonnegative by \eqref{sign}. Both vanish. The coefficient of $\|\tau_L\|^q$ is strictly positive at every point, so continuity yields $\tau_L=0$. The coefficient need not be bounded away from zero.
\end{proof}

The sign in \eqref{sign} depends on derivatives of the map and does not follow from $\Ric\ge0$. Consequently this corollary is a conditional consequence of the stress identity, not a vanishing theorem based on a domain curvature sign alone.

\begin{corollary}[Surfaces of constant scalar curvature]\label{surface-cor}
Under the assumptions of Corollary \ref{global-lemma}, if $m=2$ and $\Scal$ is constant, then
\begin{equation}\label{surface}
\frac{pq-2}{2}(\Scal-2\lambda)\int_M\|\tau_L\|^q\ddv=0.
\end{equation}
In particular, $\Scal\ne2\lambda$ implies $L$-harmonicity, without \eqref{sign}.
\end{corollary}

\begin{proof}
Use \eqref{B-surface} with $B=\|\tau_L\|^q/q$ and $d=pq$, and multiply by $q$. Since $pq>2$, the conclusion follows.
\end{proof}
If $\Scal$ is constant on a surface, \eqref{soliton-basic} gives $\Scal\nabla f=0$. When $\Scal\ne0$, $f$ is constant and $\Scal=2\lambda$. Thus the branch $\Scal\ne2\lambda$ in the surface corollaries necessarily has $\Scal=0$ and $\lambda\ne0$. The corollaries concern this flat branch; Example \ref{counterexample} shows why the Einstein branch must be excluded. The expanded $(p,q)$ and $p=4,q=2$ formulas, with sufficient derivative conditions, are given in Appendix \ref{special}.

\subsection{The cigar soliton}\label{cigar-section}
For the cigar, the metric variation gives a second-order criterion and includes $1<q<2$. The case $a=1,q=2$ is the conclusion of \cite[Corollary 1.8 in the arXiv version]{r18}.

\begin{corollary}[Weighted tension powers on the cigar]\label{cigar-main}
Equip $M=\mathbb R^2$ with Hamilton's cigar soliton
\begin{equation}\label{cigar-metric}
g=\frac{dx^2+dy^2}{1+x^2+y^2},\qquad
f=-\log(1+x^2+y^2).
\end{equation}
Let $L(y,r)=a(y)r$, where $a>0$ is smooth, and let $q>1$. Suppose $\phi:(M,g)\to(N,h)$ is a smooth critical point of $E_{2,L,q}$, $a$ and $\|\gradN a\|$ are bounded on $\phi(M)$, and
\begin{equation}\label{cigar-sufficient}
\int_M(u^q+v^q+u^{2q})\ddv<\infty.
\end{equation}
Then $\tau_L=0$. If $a$ is constant, the term $u^{2q}$ can be omitted. No sign condition on the sectional curvature of $N$ is required.
\end{corollary}

We first prove a criterion stated directly in terms of the tension field.
\begin{proposition}\label{cigar-integral-prop}
Let $(M,g,f)$ be the cigar soliton in \eqref{cigar-metric}, let $L(y,r)=a(y)r$ with smooth $a>0$, and let $q>1$. Suppose $\phi$ is critical for $E_{2,L,q}$. Write $u=\|d\phi\|$, $H=\nabla(a(\phi)d\phi)$ and $T=\tau_L$. If
\begin{equation}\label{cigar-integrability}
\int_M\left[\|T\|^q+
\bigl(\|H\|+a(\phi)u+u^2\|(\gradN a)_\phi\|\bigr)\|T\|^{q-1}\right]\ddv<\infty,
\end{equation}
then $\tau_L=0$.
\end{proposition}
\begin{proof}
Put $X=\nabla f$ and $\sigma=-\Scal/2$. For the displayed normalization of the cigar,
\[
\nabla^2f=\sigma g,\qquad
\Scal=\frac4{1+x^2+y^2}>0,
\qquad \|X\|\le2.
\]
In particular $\sigma$ and $\Ric$ are bounded. Take $\psi_R=\eta_R^2$, where $\eta_R$ satisfies \eqref{cutoffs}, and put $\Omega_R=B_{\gamma R}(o)\setminus B_R(o)$. The contracted Bianchi and soliton identities give
$\nabla\Delta f+2\Ric(\nabla f)=0$. Consequently, for
$k=\mathcal L_{\psi_RX}g$,
\[
k_{ij}=2\psi_R\sigma g_{ij}+X_i\nabla_j\psi_R+X_j\nabla_i\psi_R,
\qquad
\xi_k=(\Delta\psi_R)X+2\sigma\nabla\psi_R.
\]
Since $L''=0$ and $m=2$, Lemma \ref{conformal-linearization} gives
\[
\delta_gT=-2\psi_R\sigma T+\mathcal E_R,
\]
where, with $A=a(\phi)$,
\begin{equation}\label{cigar-error}
\begin{aligned}
\mathcal E_R={}&-(X_i\nabla_j\psi_R+X_j\nabla_i\psi_R)H_{ij}\\
&-A\,d\phi\bigl((\Delta\psi_R)X+2\sigma\nabla\psi_R\bigr)
+\langle d\phi(X),d\phi(\nabla\psi_R)\rangle(\gradN a)_\phi.
\end{aligned}
\end{equation}
Thus $\mathcal E_R$ is supported in $\Omega_R$ and
\[
\|\mathcal E_R\|\le\frac CR
\bigl(\|H\|+Au+u^2\|(\gradN a)_\phi\|\bigr)\mathbf1_{\Omega_R}.
\]
For every $q>1$, the map $T\mapsto\|T\|^q/q$ is $C^1$. The unintegrated variation and
$\operatorname{tr}k=4\psi_R\sigma+2\langle X,\nabla\psi_R\rangle$ yield
\[
0=\frac{q-1}{q}\int_M\psi_R\Scal\,\|T\|^q\ddv
+\int_M\left[\langle\|T\|^{q-2}T,\mathcal E_R\rangle
+\frac1q\|T\|^q\langle X,\nabla\psi_R\rangle\right]\ddv.
\]
The last integral tends to zero by \eqref{cigar-integrability}. Hence
\begin{equation}\label{cigar-vanishing-id}
(q-1)\int_M\Scal\,\|\tau_L\|^q\ddv=0,
\end{equation}
and positivity of $\Scal$ proves the assertion. No derivative of $W$ was used.
\end{proof}

\begin{proof}[Proof of Corollary \ref{cigar-main}]
Boundedness of $a$ and $\|\gradN a\|$ gives
\[
\|T\|+\|H\|\le C(v+u^2),\qquad
a(\phi)u+u^2\|(\gradN a)_\phi\|\le C(u+u^2).
\]
Young's inequality shows that \eqref{cigar-sufficient} implies \eqref{cigar-integrability}. If $a$ is constant, the $u^2$ terms are absent, so $\int(u^q+v^q)<\infty$ suffices. Proposition \ref{cigar-integral-prop} applies in both cases.
\end{proof}
For $q=2$ and $a=1$, this is the classical conclusion of \cite[Corollary 1.8]{r18}, with the usual finite-energy sufficient condition. The calculation above records the target weight and the range $q>1$ in the notation of this paper. It does not impose a sign on a map-dependent curvature contraction.

\section{Euclidean dilation and nonhomogeneous Lagrangians}\label{euclid}
The Euclidean argument uses the metric variation before integrating derivatives onto $W_B$. It therefore needs only the composite $C^1$ regularity of the density, and includes tension powers $1<q<2$.

Put $u=\|d\phi\|$, $v=\|\nabla d\phi\|$, and define
\begin{equation}\label{general-GP}
\begin{aligned}
G_L={}&\bigl(|L_\phi'|+u^2|L_\phi''|+u^4|L_\phi'''|\bigr)v
+u^2\|(\gradN L')_\phi\|\\
&+u^4\|(\gradN L'')_\phi\|+\|(\gradN L)_\phi\|,\\
P_L={}&|L_\phi'|u+|L_\phi''|u^3.
\end{aligned}
\end{equation}
For smooth $L$, these are defined by evaluation at $(\phi,e)$. For the homogeneous family with the regularity in Section \ref{intro}, products containing a singular radial derivative at $u=0$ are interpreted by their continuous weighted extensions. In particular, terms whose coefficients vanish at $p=2$ or $p=4$ are removed before evaluating a negative power.

\begin{theorem}[General autonomous density]\label{general-euclidean}
Let $\phi:\mathbb R^m\to N$ be a smooth critical point of $E_{B,L}$ for a smooth autonomous $L$. Suppose the composite density is $C^1$ as specified in Section \ref{variation},
\begin{equation}\label{general-euclidean-finite}
\int_{\mathbb R^m}
\bigl(|B|+e|B_\phi'|+\|\tau_{\widehat L}\|\,\|W_B\|\bigr)\,dx<\infty,
\end{equation}
and, for $\Omega_R=B_{2R}\setminus B_R$,
\begin{equation}\label{general-euclidean-annular}
\lim_{R\to\infty}\int_{\Omega_R}\|W_B\|
\bigl(G_L+R^{-1}P_L\bigr)\,dx=0.
\end{equation}
Then
\begin{equation}\label{general-euclidean-id}
\int_{\mathbb R^m}
\bigl(mB-2eB_\phi'-\langle\tau_{\widehat L},W_B\rangle\bigr)\,dx=0.
\end{equation}
The same conclusion holds for the weighted homogeneous extensions described after \eqref{general-GP}.
\end{theorem}
\begin{proof}
Choose $\eta_R(x)=\eta(x/R)$, where $0\le\eta\le1$, $\eta\in C_c^\infty(B_2)$ and $\eta=1$ on $B_1$. Then
\begin{equation}\label{euclid-cutoff}
\|\nabla\eta_R\|\le C/R,\qquad
\|\nabla^2\eta_R\|+|\Delta\eta_R|\le C/R^2.
\end{equation}
Set $X_R=\eta_Rx$. Autonomy and criticality give $\delta_gE_{B,L}(\mathcal L_{X_R}g)=0$. If $k=\mathcal L_{X_R}g$, then
\begin{equation}\label{euclid-k}
\begin{aligned}
k_{ij}&=2\eta_R\delta_{ij}+x_i\partial_j\eta_R+x_j\partial_i\eta_R,\\
\operatorname{tr}k&=2m\eta_R+2\langle x,\nabla\eta_R\rangle,\\
\xi_k&=x\Delta\eta_R+2\nabla\eta_R.
\end{aligned}
\end{equation}
Write $A=L_\phi'$, $C_0=L_\phi''$, $H=\nabla(A\,d\phi)$ and
$F_R=\langle d\phi(x),d\phi(\nabla\eta_R)\rangle$. Formula \eqref{linearize-metric} gives
\begin{equation}\label{euclid-error-def}
\delta_g\tau_L=-\eta_R\tau_{\widehat L}+\mathcal E_R,
\end{equation}
where
\begin{equation}\label{euclid-error}
\begin{aligned}
\mathcal E_R={}&-2(A+eC_0)d\phi(\nabla\eta_R)
-A(\Delta\eta_R)d\phi(x)\\
&-(x_i\partial_j\eta_R+x_j\partial_i\eta_R)H_{ij}
-\divg(C_0F_Rd\phi)
+F_R(\gradN L')_\phi.
\end{aligned}
\end{equation}
All terms are supported in $\Omega_R$. On this annulus,
\[
|F_R|\le Cu^2,\qquad
\|\nabla F_R\|\le C(uv+R^{-1}u^2),\qquad
\|\nabla C_0\|\le |L_\phi'''|uv+\|(\gradN L'')_\phi\|u.
\]
Also $\|H\|\le C[(|A|+u^2|C_0|)v+u^2\|(\gradN L')_\phi\|]$. Expanding only the indicated divergence yields
\begin{equation}\label{general-error-bound}
\|\mathcal E_R\|\le C(G_L+R^{-1}P_L)\mathbf1_{\Omega_R},
\qquad \|\tau_L\|+\|\tau_{\widehat L}\|\le CG_L.
\end{equation}
For the second estimate use
$\widehat L'=2L'+2rL''$, $\widehat L''=4L''+2rL'''$ and
$\gradN\widehat L=2r\gradN L'$. These estimates involve at most two covariant derivatives of $\phi$.

Since $\delta_ge=-2\eta_Re-F_R$, the unintegrated metric variation becomes
\begin{equation}\label{B-pohozaev-local}
\begin{aligned}
0={}&\int\eta_R
\bigl(mB-2eB_\phi'-\langle\tau_{\widehat L},W_B\rangle\bigr)\,dx\\
&+\int_{\Omega_R}\bigl[B\langle x,\nabla\eta_R\rangle
-B_\phi'F_R+\langle W_B,\mathcal E_R\rangle\bigr]\,dx.
\end{aligned}
\end{equation}
The first two boundary terms tend to zero by \eqref{general-euclidean-finite}; the last does so by \eqref{general-euclidean-annular}. Dominated convergence proves the assertion. In the homogeneous family, differentiating the weighted products in \eqref{euclid-error} gives their continuous extensions at allowed zeros, so the same calculation applies there.
\end{proof}

\begin{corollary}[A general tension power]\label{general-power-euclidean}
Let $q>1$ and let $\phi$ be critical for $E_{2,L,q}$ on $\mathbb R^m$. Suppose
\begin{equation}\label{general-power-assumptions}
\int_{\mathbb R^m}G_L^q\,dx<\infty,
\qquad
\sup_{R\ge1}R^{-q}\int_{\Omega_R}P_L^q\,dx<\infty.
\end{equation}
Then
\begin{equation}\label{general-power-id}
\frac mq\int_{\mathbb R^m}\|\tau_L\|^q\,dx
=\int_{\mathbb R^m}\langle\tau_{\widehat L},
\|\tau_L\|^{q-2}\tau_L\rangle\,dx.
\end{equation}
If, along the map, constants $p_-\le p_+$ satisfy
\begin{equation}\label{tension-scaling-bounds}
p_-\|\tau_L\|^2\le
\langle\tau_{\widehat L},\tau_L\rangle
\le p_+\|\tau_L\|^2,
\end{equation}
then $m<qp_-$ or $m>qp_+$ implies $\tau_L=0$.
\end{corollary}
\begin{proof}
By \eqref{general-error-bound}, $\tau_L,\tau_{\widehat L}\in L^q$, so the interior integrals are finite. H\"older's inequality bounds the two annular terms by
\[
\left(\int_{\Omega_R}\|\tau_L\|^q\right)^{(q-1)/q}
\left(\int_{\Omega_R}G_L^q\right)^{1/q},
\qquad
\left(\int_{\Omega_R}\|\tau_L\|^q\right)^{(q-1)/q}
\left(R^{-q}\int_{\Omega_R}P_L^q\right)^{1/q}.
\]
Both tend to zero. Apply Theorem \ref{general-euclidean}; multiplying \eqref{tension-scaling-bounds} by $\|\tau_L\|^{q-2}$ gives the last assertion. Products are interpreted continuously as zero at $\tau_L=0$.
\end{proof}

\begin{corollary}[A controlled failure of homogeneity]\label{euclidean-defect-cor}
Under the hypotheses of Corollary \ref{general-power-euclidean}, fix a real number $p$ and put
\begin{equation}\label{Dp-definition}
D_p(y,r)=2rL'(y,r)-pL(y,r).
\end{equation}
Then
\begin{equation}\label{euclidean-defect-id}
\left(\frac mq-p\right)\int\|\tau_L\|^q\,dx
=\int\langle\tau_{D_p},\|\tau_L\|^{q-2}\tau_L\rangle\,dx.
\end{equation}
In particular, if $\|\tau_{D_p}\|\le\varepsilon\|\tau_L\|$ and
$|m-pq|>q\varepsilon$, then $\tau_L=0$.
\end{corollary}
\begin{proof}
The map $L\mapsto\tau_L$ is linear, so
$\tau_{\widehat L}=p\tau_L+\tau_{D_p}$. Substitute this into \eqref{general-power-id}. The absolute value of the right side of \eqref{euclidean-defect-id} is at most $\varepsilon\int\|\tau_L\|^q$.
\end{proof}
The comparison in \eqref{tension-scaling-bounds} and the defect bound concern the map as well as $L$. Convexity and scalar growth indices of $L$ alone do not imply them; see Example \ref{jet-example}.

\subsection{Homogeneous Lagrangians}
\begin{theorem}[Pohozaev identity]\label{B-euclid}
Let $\phi:\mathbb R^m\to N$ be a smooth critical point of \eqref{B-functional}, where $L$ is given by \eqref{homogeneous-L} and $p$ has the regularity specified in Section \ref{intro}. Assume $a$ and $\|\gradN a\|$ are bounded on the image. Suppose the composite density has the $C^1$ regularity specified after \eqref{B-aux},
\begin{equation}\label{B-finite}
\int_{\mathbb R^m}(|B|+e|B_\phi'|+s|B_\tau'|)\,dx<\infty,
\end{equation}
and
\begin{equation}\label{B-annular}
\lim_{R\to\infty}\int_{\Omega_R}\|W_B\|
\bigl(G+R^{-1}u^{p-1}\bigr)\,dx=0.
\end{equation}
Then
\begin{equation}\label{B-euclid-id}
\int_{\mathbb R^m}(mB-2eB_\phi'-2psB_\tau')\,dx=0.
\end{equation}
No target curvature sign is required.
\end{theorem}

\begin{proof}
Here $\widehat L=pL$, $\langle\tau_{\widehat L},W_B\rangle=2psB_\tau'$, and the bounds on $a,\gradN a$ give $G_L\le CG$ and $P_L\le Cu^{p-1}$, with continuous weighted extensions as above. Consequently \eqref{general-error-bound} becomes
\begin{equation}\label{euclid-error-bound}
\|\mathcal E_R\|\le C(G+R^{-1}u^{p-1})\mathbf1_{\Omega_R}.
\end{equation}
The hypotheses imply \eqref{general-euclidean-finite} and \eqref{general-euclidean-annular}. Theorem \ref{general-euclidean} proves the identity.
\end{proof}

\begin{corollary}[Bounds on the scaling derivative]\label{B-degree-bounds}
Under the assumptions of Theorem \ref{B-euclid}, suppose $B\ge0$, its zero set along the map implies $s=0$, and
\begin{equation}\label{B-degree-interval}
d_-B\le2eB_\phi'+2psB_\tau'\le d_+B
\end{equation}
for constants $d_-\le d_+$. If $m<d_-$ or $m>d_+$, then $\tau_L=0$.
\end{corollary}
\begin{proof}
For $m<d_-$, the integrand in \eqref{B-euclid-id} is bounded above by $(m-d_-)B$. For $m>d_+$ it is bounded below by $(m-d_+)B$. In either case $\int B=0$, and continuity and the zero-set assumption give the result.
\end{proof}
For example, if $B=b(y)F(s)$ with $b>0$, $F(0)=0$, $F(s)>0$ for $s>0$, and
$\theta_-\le sF'(s)/F(s)\le\theta_+$, the critical interval is $[2p\theta_-,2p\theta_+]$. This criterion does not require exact homogeneity or convexity of $F$.
If $F$ is convex, then $sF'(s)\ge F(s)$; hence $m<2p$ already gives vanishing under the analytic assumptions of Theorem \ref{B-euclid}, without an upper growth bound on $F$.

\subsection{Power densities}
\begin{theorem}[Tension powers on Euclidean space]\label{B-powers}
Let $L$ be given by \eqref{homogeneous-L}, let $\phi:\mathbb R^m\to(N,h)$ be smooth, let $1<q_1\le\cdots\le q_k$, and let $b_j$ be smooth positive functions bounded above on $\phi(\mathbb R^m)$. Suppose $\phi$ is critical for
\begin{equation}\label{B-power-energy}
E_{B,L}(\phi)=\sum_{j=1}^k E_j,\qquad
E_j=\frac1{q_j}\int_{\mathbb R^m}b_j(\phi)\|\tau_L\|^{q_j}\,dx.
\end{equation}
Assume the regularity conventions above and that $a$ and $\|\gradN a\|$ are bounded on the image. For every $j$, suppose
\begin{equation}\label{B-power-second}
\int_{\mathbb R^m}G^{q_j}\,dx<\infty
\end{equation}
and
\begin{equation}\label{B-power-growth}
\sup_{R\ge1}R^{-q_j}\int_{\Omega_R}u^{q_j(p-1)}\,dx<\infty.
\end{equation}
Then
\begin{equation}\label{B-power-id}
\sum_{j=1}^k(m-pq_j)E_j=0.
\end{equation}
If all $q_j=q$, then $m\ne pq$ implies $\tau_L=0$. If $q_1<q_k$, then $m\le pq_1$ or $m\ge pq_k$ implies $\tau_L=0$, including both endpoints. No bound on $\gradN b_j$ is assumed.
\end{theorem}

For a single power, the coefficient $m-pq$ is determined by the scaling $\tau_L^{cg}=c^{-p/2}\tau_L^g$. Criticality is required for the sum in \eqref{B-power-energy}; criticality for one summand alone does not suffice.

\begin{proof}[Proof of Theorem \ref{B-powers}]

By \eqref{basic-bounds-eq}, $\|\tau_L\|\le CG$. Thus the energies are finite, $B_\phi'=0$, and
\[
W_B=\sum_j b_j(\phi)\|\tau_L\|^{q_j-2}\tau_L,
\qquad 2psB_\tau'=\sum_j p\,b_j(\phi)\|\tau_L\|^{q_j}.
\]
H\"older's inequality bounds the contribution of $G$ in \eqref{B-annular} by a sum of
\[
C\left(\int_{\Omega_R}\|\tau_L\|^{q_j}\,dx\right)^{(q_j-1)/q_j}
\left(\int_{\Omega_R}G^{q_j}\,dx\right)^{1/q_j},
\]
which tends to zero. The remaining contribution is bounded by a sum of
\begin{equation}\label{B-growth-holder}
C\left(R^{-q_j}\int_{\Omega_R}u^{q_j(p-1)}\,dx\right)^{1/q_j}
\left(\int_{\Omega_R}\|\tau_L\|^{q_j}\,dx\right)^{(q_j-1)/q_j},
\end{equation}
which also tends to zero. Apply Theorem \ref{B-euclid} to obtain \eqref{B-power-id}.
Outside the closed interval $[pq_1,pq_k]$, all nonzero coefficients have the same sign. At either endpoint, at least one coefficient remains nonzero when $q_1<q_k$; its energy must vanish. Positivity of the corresponding $b_j$ then gives $\tau_L=0$. The calculation only differentiates the metric, so derivatives of $b_j$ do not enter the estimates.
\end{proof}

The global assumption $u^{q(p-1)}\in L^1$ implies \eqref{B-power-growth}, but the reverse implication fails even among smooth critical maps; see Example \ref{stereo-example}. This does not remove the second-order assumption \eqref{B-power-second}.

\begin{corollary}[The single power and $p=4$]\label{euclid-thm}\label{gaussian-cor}
Under the assumptions of Theorem \ref{B-powers} with one exponent $q>1$ and $b_1=1$,
\begin{equation}\label{euclid-id}
(m-pq)\int_{\mathbb R^m}\|\tau_L\|^q\,dx=0.
\end{equation}

In particular, for $a=1$ and $p=4$, it is sufficient that
\begin{equation}\label{four-gaussian}
\int_{\mathbb R^m}u^{2q}v^q\,dx<\infty,
\qquad
\sup_{R\ge1}R^{-q}\int_{\Omega_R}u^{3q}\,dx<\infty.
\end{equation}
If $m\ne4q$, then $\tau_4=0$. At $q=2$, the excluded dimension is $8$.
\end{corollary}
\begin{proof}
Apply \eqref{B-power-id} with $k=1$. For $a=1,p=4$, $G=u^2v$, giving \eqref{four-gaussian}.
\end{proof}
These conclusions also apply when Euclidean space is regarded as the Gaussian soliton with $f=\lambda\|x\|^2/2$; their proof is independent of $\lambda$.

\begin{remark}[A decay test for the hypotheses]
For constant $a$, suppose $u=O((1+|x|)^{-\sigma})$ and $v=O((1+|x|)^{-\sigma-1})$, with $\sigma\ge0$. The one-power assumptions \eqref{B-power-second}--\eqref{B-power-growth} hold whenever
\[
q\bigl((p-1)\sigma+1\bigr)>m.
\]
Indeed, $G^q=O((1+|x|)^{-q((p-1)\sigma+1)})$, and the annular expression is $O(R^{m-q((p-1)\sigma+1)})$. By comparison, the bound on $u$ alone guarantees $u^{q(p-1)}\in L^1$ only when $q(p-1)\sigma>m$. These are sufficient decay tests, not necessary decay rates for critical maps.
\end{remark}

For $a\equiv1$, $p=4$, $q_1=2$, $q_2=4$ and $b_1=b_2=1$, the identity is
\begin{equation}\label{B-four-example}
\frac{m-8}{2}\int_{\mathbb R^m}\|\tau_4\|^2\,dx
+\frac{m-16}{4}\int_{\mathbb R^m}\|\tau_4\|^4\,dx=0.
\end{equation}
Under \eqref{B-power-growth} and
$\int(u^4v^2+u^8v^4)\,dx<\infty$, this forces $4$-harmonicity for $m\le8$ or $m\ge16$. It concerns critical points of the sum of the two energies. A critical point of $E_{4,2}$ alone need not be critical for that sum.

\begin{corollary}[Adding a first-order energy]\label{B-first-order}
Let $k\ge2$, $q>1$ and $\mu>0$, and consider
\begin{equation}\label{B-mixed-energy}
B(y,r,s)=\frac\mu k(2r)^{k/2}+\frac1q(2s)^{q/2}.
\end{equation}
Suppose $\phi$ is critical, $L$ satisfies the hypotheses of Theorem \ref{B-euclid}, and
$\int u^k\,dx<\infty$, together with \eqref{B-power-second} and \eqref{B-power-growth} for the single exponent $q$. Then
\begin{equation}\label{B-mixed-id}
\frac{\mu(m-k)}k\int_{\mathbb R^m}u^k\,dx
+\frac{m-pq}q\int_{\mathbb R^m}\|\tau_L\|^q\,dx=0.
\end{equation}
If $m<\min\{k,pq\}$ or $m>\max\{k,pq\}$, then $\phi$ is constant. At $m=k\ne pq$, it is $L$-harmonic; at $m=pq\ne k$, it is constant.
\end{corollary}
\begin{proof}
Here $2eB_\phi'=\mu u^k$, while $W_B=\|\tau_L\|^{q-2}\tau_L$. The proof of Theorem \ref{B-powers} verifies \eqref{B-annular}; the extra terms in \eqref{B-finite} are integrable by the assumption on $u^k$. Formula \eqref{B-euclid-id} becomes \eqref{B-mixed-id}. The stated conclusions follow from its signs.
\end{proof}
For $L=r$ and $k=q=2$, this is the positive-coefficient interpolating sesqui-harmonic functional, up to normalization \cite{r26,r27}. The classification in \cite[Theorem 3.2]{r27} treats this classical case under its finite-energy assumptions. Corollary \ref{B-first-order} gives the corresponding identity for the homogeneous $L$-tension field and the indicated powers.

\section{Examples and limitations}\label{example}
\begin{example}[The critical dimension and the annular hypothesis]\label{stereo-example}
For $p=q=2$ and $a=1$, the dimension restriction $m\ne4$ cannot be removed from Theorem \ref{B-powers}. Consider the inverse stereographic projection
\[
\phi:\mathbb R^4\longrightarrow S^4,\qquad
\phi(x)=\frac{(2x,\|x\|^2-1)}{1+\|x\|^2}.
\]
This is a proper biharmonic map, as proved by Loubeau and Ou \cite{r28}. We verify its relation to our analytic assumptions. Put $\rho=\|x\|$ and $\omega=2/(1+\rho^2)$. Since $\phi^*h=\omega^2g_{\mathrm{eucl}}$, the difference of the induced connection and the Euclidean connection has coefficients
\[
\Gamma^k_{ij}=\delta^k_j\partial_i\log\omega
+\delta^k_i\partial_j\log\omega
-\delta_{ij}\partial_k\log\omega.
\]
Taking norms and traces gives
\begin{equation}\label{stereo-norms}
\|d\phi\|^2=\frac{16}{(1+\rho^2)^2},\qquad
\|\nabla d\phi\|^2=\frac{160\rho^2}{(1+\rho^2)^4},\qquad
\|\tau(\phi)\|^2=\frac{64\rho^2}{(1+\rho^2)^4}.
\end{equation}
For the second expression, use $\sum_{i,j,k}(\Gamma^k_{ij})^2=10\|\nabla\log\omega\|^2$ and the target norm factor $\omega^2$. Polar integration therefore gives
\[
\int_{\mathbb R^4}\|\nabla d\phi\|^2\,dx<\infty,
\qquad
\int_{\mathbb R^4}\|d\phi\|^2\,dx=\infty,
\qquad
\int_{B_{2R}\setminus B_R}\|d\phi\|^2\,dx\le C\quad(R\ge1).
\]
Thus \eqref{B-power-second} and \eqref{B-power-growth} hold, although the corresponding global first-derivative condition fails. The tension is not zero. This known critical map demonstrates a distinction between the two sets of analytic hypotheses; it does not establish sharpness for arbitrary $p$ and $q$.
\end{example}

\begin{example}[A zero-energy critical map]\label{B-zero-example}
The zero-set assumption on $B$ cannot be omitted. Let $M=\mathbb R^m$, $N=\mathbb R$, $L=r$ and $B(s)=(s-\tfrac12)^2$. The map
$\phi(x)=\|x\|^2/(2m)$ has $\tau_L=\Delta\phi=1$ and $s=1/2$. Consequently $B=0$ and $B_\tau'=0$ along the map. The first variation vanishes for every compactly supported variation, and $E_{B,L}=0$, but $\phi$ is not $L$-harmonic. In this example \eqref{B-finite} and \eqref{B-annular} hold because all their integrands vanish. Thus the Euclidean identity is valid but yields no vanishing of $\tau_L$.
\end{example}

\begin{example}[The Einstein surface branch]\label{counterexample}

There exists a complete noncompact finite-area two-dimensional gradient Ricci soliton with $\lambda=-2$ and $\operatorname{Scal}^M=-4=2\lambda$, together with a smooth $4$-biharmonic map satisfying \eqref{four-I} and \eqref{four-deriv-cond} that is not $4$-harmonic.

\end{example}
\begin{proof}Take a complete finite-area hyperbolic surface with cusps, rescaling its metric to have constant Gaussian curvature $-2$; write it as $(M^2,g)$. Set $f=0$. Then $\operatorname{Ric}^M=-2g$, so \eqref{soliton} holds. On $N=I\times M$, use the warped product metric

\begin{equation}\label{warp}
\begin{aligned}
I&=(-1/8,1/8),\qquad h=dt^2+w(t)^2g,\\
w(t)&=(1+4t)^{1/4},\qquad \phi(x)=(0,x).
\end{aligned}
\end{equation}

Along the slice, let $\nu=\partial_t$. Since $w(0)=1$, $w^{\prime}(0)=1$ and $w^{\prime\prime}(0)=-3$, the warped product connection gives

\begin{equation}\label{warp-connection}
\begin{aligned}
\|d\phi\|^2&=2,\qquad
(\nabla d\phi)(e_i,e_j)=-\delta_{ij}\nu,\qquad
\nabla_{e_i}\nu=d\phi(e_i),\\
\tau(\phi)&=-2\nu,\qquad \tau_4(\phi)=-4\nu,\\
\nabla_{e_i}\tau_4(\phi)&=-4d\phi(e_i),\qquad
\langle d\phi,\nabla\tau_4(\phi)\rangle=-8.
\end{aligned}
\end{equation}

In particular, $\tau_4(\phi)\ne0$. Choose a normal frame at a fixed point and compute the three terms of \eqref{euler} for $a\equiv1$, $p=4,q=2$:

\begin{equation}\label{warp-euler}
\begin{aligned}
\operatorname{div}_g(\|d\phi\|^2\nabla\tau_4(\phi))&=16\nu,\\
2\operatorname{div}_g(\langle d\phi,\nabla\tau_4(\phi)\rangle d\phi)&=32\nu,\\
\|d\phi\|^2\sum_iR^N(\tau_4(\phi),d\phi(e_i))d\phi(e_i)&=-48\nu.
\end{aligned}
\end{equation}

The last formula uses $R^N(\nu,d\phi(e_i))d\phi(e_i)=-w^{\prime\prime}(0)\nu=3\nu$, with no summation over the fixed index $i$. Hence $\tau_{2,4}=-16\nu-32\nu+48\nu=0$. Moreover,

\begin{equation}\label{warp-derivatives}
\begin{aligned}
\|\nabla d\phi\|^2&=2,\\
(\nabla\nabla d\phi)(e_k,e_i,e_j)&=-\delta_{ij}d\phi(e_k),\\
\|\nabla\nabla d\phi\|^2&=4.
\end{aligned}
\end{equation}

All relevant quantities are bounded, and $\operatorname{Vol}_g(M)<\infty$, so \eqref{four-I} and \eqref{four-deriv-cond} hold. The example lies precisely on the Einstein branch $\operatorname{Scal}^M=2\lambda$, excluded by Corollary \ref{surface-cor}.

\end{proof}

The constant-tension-norm observation also appears in \cite[Corollary 7]{r22}. Here, since $\|\tau_4(\phi)\|=4$ is constant, $W=4^{q-2}\tau_4(\phi)$ for every $q>1$, and all terms of \eqref{euler} are multiplied by the same constant. Thus the example is also a proper $(4,q)$-biharmonic map satisfying the corresponding $(A_q)$ and $(B_q)$. If a complete target is desired, extend $w$ smoothly away from a neighborhood of zero to a positive function on $\mathbb R$ bounded below by a positive constant. All calculations along the slice remain unchanged.

\begin{example}[Scalar growth does not control tension scaling]\label{jet-example}
Consider $L(r)=r+r^2$ and scalar maps on $\mathbb R^2$. Then $L'>0$, $L''>0$, and
$1\le rL'(r)/L(r)\le2$ for $r>0$. At a point, prescribe
\[
\nabla\phi=(1,0),\qquad
\nabla^2\phi=\begin{pmatrix}1&0\\0&-2+\varepsilon/2\end{pmatrix}.
\]
These jets are realized by a quadratic polynomial. At that point $e=1/2$, and
\[
\tau_L=(1+2e)\Delta\phi+2\nabla^2\phi(\nabla\phi,\nabla\phi)=\varepsilon,
\qquad
\tau_{\widehat L}=2+3\varepsilon,
\]
because $\widehat L=2r+4r^2$. For $\varepsilon\ne0$,
\begin{equation}\label{jet-ratio}
\frac{\langle\tau_{\widehat L},\tau_L\rangle}{\|\tau_L\|^2}
=3+\frac2\varepsilon.
\end{equation}
Thus no two-sided pointwise bound of the form \eqref{tension-scaling-bounds} follows from those scalar conditions alone. This is a local obstruction to that implication, not an assertion that the polynomial is a global critical point of $E_{2,L,q}$.
\end{example}

\begin{example}[A nonhomogeneous obstruction]\label{nonhom-example}
For an arbitrary positive autonomous Lagrangian, finite $E_{2,L,q}$ does not force $L$-harmonicity, even on a closed negative Einstein domain. Let $M$ be a closed hyperbolic manifold with $\Ric=\lambda g$, $\lambda<0$, and $f=0$. Take the complete target $N=\mathbb R$ and
\begin{equation}\label{potential-L}
L(y,r)=e^r+2+\arctan y.
\end{equation}
This function is smooth and positive for all $(y,r)\in\mathbb R^2$. The constant map $\phi=0$ is a critical point of $E_{2,L,q}$ for every $q>1$ but is not $L$-harmonic.
\end{example}
\begin{proof}
At $\phi=0$, $d\phi=0$, $L_\phi'=1$, $(\gradN L)_\phi=1$, and hence $\tau_L=-1$. Since $\partial_y^2L(0,0)=0$, a variation with field $v$ has $\delta_\phi\tau_L=\Delta v$. Thus
\[
\delta_\phi E_{2,L,q}=-\int_M\Delta v\ddv=0.
\]
The energy equals $\operatorname{Vol}(M)/q$ and is finite. The same construction works on a complete noncompact finite-volume hyperbolic manifold, since variations are compactly supported. In fact, the quantities in \eqref{IL} are finite, while both curvature contractions in \eqref{global-id} vanish. Inserting this example into a putative homogeneous formula with any $p>0$ would give the strictly positive first coefficient $\lambda(m-pq)-m\lambda=-pq\lambda$. The missing homogeneity terms are therefore essential. In the correct general identity, $\tau_{\widehat L}=0$ and $m\lambda-\Scal=0$, so \eqref{general-power-soliton-id} is satisfied. This example also satisfies \eqref{general-power-soliton-I}.
\end{proof}

\begin{remark}[The precise homogeneity defect]\label{defect-rem}
For a general autonomous $L(y,r)$ and a fixed comparison parameter $p$, direct tracing of \eqref{stress} gives
\begin{equation}\label{trace-defect}
\begin{aligned}
\operatorname{tr}_gS_{2,L,q}
={}&\left(\frac mq-p\right)\|\tau_L\|^q+(p-m)\divg\alpha+\mathcal R_{L,p,q},\\
\mathcal R_{L,p,q}
={}&[2eL_\phi''-(p-2)L_\phi']\langle d\phi,\nabla W\rangle\\
&+\langle2e(\gradN L')_\phi-p(\gradN L)_\phi,W\rangle.
\end{aligned}
\end{equation}
The defect vanishes for \eqref{homogeneous-L}. With $D_p$ from \eqref{Dp-definition}, put $\alpha_{D_p}=(D_p)'_\phi\langle d\phi,W\rangle$. Then
\[
\mathcal R_{L,p,q}=\divg\alpha_{D_p}-\langle\tau_{D_p},W\rangle.
\]
Thus the scaling-field formulation and the trace-defect formulation agree after integration whenever their divergence term has vanishing boundary flux. The following alternative records the earlier integrability formulation. Assume now that $q\ge2$, the domain is a complete gradient Ricci soliton with bounded Ricci curvature, and $\phi$ is a smooth critical point for the general autonomous $L$. Suppose the fields in Lemma \ref{variation-lemma} have the stated regularity, \eqref{IL} holds, and $\mathcal R_{L,p,q}\in L^1(M)$. Under these hypotheses the same cutoff calculation yields \eqref{global-id} with the additional summand $q\lambda\int_M\mathcal R_{L,p,q}\ddv$ on its right-hand side. It has no prescribed sign. In Example \ref{nonhom-example}, $\mathcal R_{L,p,q}=p$, exactly cancelling the term $-pq\lambda\operatorname{Vol}(M)$.
\end{remark}

\appendix
\section{Map variation and comparison of conventions}\label{map-appendix}
Let $q\ge2$, and let $W$ and $\Psi_{L,q}$ be as in \eqref{W}--\eqref{Psi}. The following formula is also valid for a general autonomous $L$ whenever its displayed fields have the required regularity.
\begin{lemma}\label{variation-lemma}
For a compactly supported variation field $V$, the first variation satisfies
$\delta_\phi E_{2,L,q}=-\langle\tau_{2,L,q},V\rangle$ in the distributional pairing, where
\begin{equation}\label{euler}
\begin{aligned}
\tau_{2,L,q}={}&-\divg(L_\phi'\nabla W)
-L_\phi'\sum_iR^N(W,d\phi(e_i))d\phi(e_i)\\
&-\divg(\Psi_{L,q}d\phi)
+\langle d\phi,\nabla W\rangle(\gradN L')_\phi
+(\nabla_W^N\gradN L)_\phi.
\end{aligned}
\end{equation}

The last target derivative holds $r$ fixed. With the metric convention
\begin{equation}\label{metric-first}
\delta_gE_{2,L,q}=\frac12\int_M\langle S_{2,L,q},k\rangle\ddv,
\end{equation}

the power stress is
\begin{equation}\label{stress}
\begin{aligned}
S_{2,L,q}(X,Y)={}&-\Bigl(\frac{q-1}{q}\|\tau_L\|^q
+\langle(\gradN L)_\phi,W\rangle
+L_\phi'\langle d\phi,\nabla W\rangle\Bigr)g(X,Y)\\
&+L_\phi'\langle d\phi(X),\nabla_YW\rangle
+L_\phi'\langle d\phi(Y),\nabla_XW\rangle\\
&+\Psi_{L,q}\langle d\phi(X),d\phi(Y)\rangle.
\end{aligned}
\end{equation}

\end{lemma}
\begin{proof}
The chain rule and variation of the pullback connection give
\begin{equation}\label{linearize-map}
\begin{aligned}
\delta_\phi\tau_L={}&\divg(L_\phi'\nabla V)
+L_\phi'\sum_iR^N(V,d\phi(e_i))d\phi(e_i)\\
&+\divg\Bigl(\bigl[L_\phi''\langle d\phi,\nabla V\rangle
+\langle(\gradN L')_\phi,V\rangle\bigr]d\phi\Bigr)\\
&-(\nabla_V^N\gradN L)_\phi
-\langle d\phi,\nabla V\rangle(\gradN L')_\phi.
\end{aligned}
\end{equation}

Pair this formula with $W$ and integrate by parts. Symmetry of the target Hessian of $L$ and the curvature symmetries yield \eqref{euler}. The stress formula is \eqref{B-stress} for $B=(2s)^{q/2}/q$, using \eqref{alpha-div}.
\end{proof}

\begin{remark}[Comparison of variation conventions]\label{variation-comparison}
For $a\equiv1$ and $q=2$, \eqref{stress} recovers the stress $p$-bienergy tensor of \cite{r20}. At $p=q=2$, our convention \eqref{metric-first} gives the negative of the tensor denoted by $S_2$ in \cite{r24}; this overall sign does not affect conservation.

For target-dependent $L$, the formulas are derived directly above. The mixed term in \eqref{linearize-map} is needed in addition to the terms printed in \cite[Eq.~(4.3)]{r5}. For $q=2$, the target-gradient contribution to the pullback-metric term in \eqref{stress} has the positive sign found in the intermediate formulas (6.21)--(6.22) of \cite{r5}, whereas its sign is negative in the displayed statement of their Theorem 6.2. These differences disappear when $L$ has no target dependence. As a direct check of the mixed map term, for $N=\mathbb R$ and $L(y,r)=a(y)r$ one has
\[
\tau_L=a(\phi)\Delta\phi+a'(\phi)e(\phi),\qquad
\delta_\phi\tau_L=a(\phi)\Delta v+a'(\phi)\langle d\phi,dv\rangle
+[a'(\phi)\Delta\phi+a''(\phi)e(\phi)]v.
\]
Formula \eqref{linearize-map} gives exactly this expression. Here $a'$ and $a''$ are ordinary derivatives of the real function $a$.
\end{remark}

\section{Power specializations and derivative criteria}\label{special}
Throughout this appendix $L$ is given by \eqref{homogeneous-L}. We record the original power notation and sufficient assumptions used in the soliton results. The Euclidean theorem needs only \eqref{B-power-second}--\eqref{B-power-growth}.

\begin{proposition}[Sufficient assumptions for a target-dependent weight]\label{AL-prop}
Suppose $a$ and $\|\gradN a\|$ are bounded on $\phi(M)$. Then \eqref{IL} follows from
\begin{equation}\label{AL}
\int_M\left(\|d\phi\|^{q(p-1)}
+\|d\phi\|^{q(p-2)}\|\nabla d\phi\|^q
+\|d\phi\|^{pq}\right)\ddv<\infty
\end{equation}
and, when $p>2$, from
\begin{equation}\label{BL}
\int_M\|d\phi\|^{p-2}|\langle d\phi,\nabla W\rangle|\ddv<\infty.
\end{equation}
Condition \eqref{BL} is unnecessary when $p=2$. If $a$ is a positive constant, the term $\|d\phi\|^{pq}$ in \eqref{AL} may be omitted.
\end{proposition}
\begin{proof}
The definitions yield
\begin{equation}\label{target-weight}
\begin{aligned}
L_\phi'&=a(\phi)\|d\phi\|^{p-2},\qquad
(\gradN L)_\phi=\frac{\|d\phi\|^p}{p}(\gradN a)\circ\phi,\\
\tau_L&=a(\phi)\tau_p
+\|d\phi\|^{p-2}d\phi(\nabla(a\circ\phi))
-\frac{\|d\phi\|^p}{p}(\gradN a)\circ\phi,\\
\Psi_{L,q}&=(p-2)a(\phi)\|d\phi\|^{p-4}\langle d\phi,\nabla W\rangle
+\|d\phi\|^{p-2}\langle(\gradN a)\circ\phi,W\rangle.
\end{aligned}
\end{equation}
Writing $u=\|d\phi\|$ and $v=\|\nabla d\phi\|$ temporarily, the product rule gives
\begin{equation}\label{weighted-bounds}
\begin{aligned}
\|\tau_L\|+\|\nabla(L_\phi'd\phi)\|&\le C(u^{p-2}v+u^p),\\
|L_\phi'|u&\le Cu^{p-1},\\
|\Psi_{L,q}|u^2&\le C(p-2)u^{p-2}|\langle d\phi,\nabla W\rangle|
+Cu^p\|\tau_L\|^{q-1}.
\end{aligned}
\end{equation}
Young's inequality bounds $u^p\|\tau_L\|^{q-1}$ by $C(u^{pq}+\|\tau_L\|^q)$, and bounds the mixed second-derivative term by $C(u^{q(p-2)}v^q+u^{pq})$. This proves the assertion. If $a$ is constant, all terms involving $\gradN a$ vanish, and $u^p$ and $u^{pq}$ can be removed from these estimates.
\end{proof}

\subsection{The \texorpdfstring{$(p,q)$}{(p,q)} and fourth-power formulas}
For $a\equiv1$, $\tau_L=\tau_p$ and $\Psi_{L,q}=(p-2)\|d\phi\|^{p-4}\langle d\phi,\nabla W\rangle$. Thus \eqref{IL} becomes
\begin{equation}\label{I}
\begin{aligned}
&\int_M\bigl(\|d\phi\|^{q(p-1)}+\|\tau_p\|^q
+\|\nabla(\|d\phi\|^{p-2}d\phi)\|\,\|\tau_p\|^{q-1}\bigr)\ddv\\
&\qquad+(p-2)\int_M\|d\phi\|^{p-2}|\langle d\phi,\nabla W\rangle|\ddv<\infty.
\end{aligned}
\end{equation}
The global identity reads
\begin{equation}\label{pq-global}
\begin{aligned}
0={}&\int_M[\lambda(m-pq)-\Scal]\|\tau_p\|^q\ddv\\
&+2q\int_M\Ric(e_i,e_j)
\langle\nabla(\|d\phi\|^{p-2}d\phi)(e_i,e_j),W\rangle\ddv\\
&-q(p-2)\int_M\|d\phi\|^{p-4}\langle d\phi,\nabla W\rangle
\Ric(e_i,e_j)\langle d\phi(e_i),d\phi(e_j)\rangle\ddv.
\end{aligned}
\end{equation}
In particular, the exponent in its first coefficient is $pq$.

A convenient sufficient pair for \eqref{I} is
\begin{equation}\label{Aq}
\int_M\bigl(\|d\phi\|^{q(p-1)}+\|d\phi\|^{q(p-2)}\|\nabla d\phi\|^q\bigr)\ddv<\infty
\qquad(A_q),
\end{equation}
\begin{equation}\label{Bq}
\int_M\|d\phi\|^{p-2}|\langle d\phi,\nabla W\rangle|\ddv<\infty
\qquad(B_q),
\end{equation}
with $(B_q)$ omitted for $p=2$. The inequalities
\begin{equation}\label{H-bounds}
\begin{aligned}
\|d\phi\|^{p-2}\|\nabla d\phi\|
&\le\|\nabla(\|d\phi\|^{p-2}d\phi)\|
\le(p-1)\|d\phi\|^{p-2}\|\nabla d\phi\|,\\
\|\tau_p\|&\le\sqrt m\,\|\nabla(\|d\phi\|^{p-2}d\phi)\|
\end{aligned}
\end{equation}
prove this implication directly. Indeed, the derivative of $v\mapsto\|v\|^{p-2}v$ has radial and orthogonal eigenvalues $(p-1)\|v\|^{p-2}$ and $\|v\|^{p-2}$, respectively.

For comparison, another sufficient condition is
\begin{equation}\label{strong}
\begin{aligned}
\int_M\Bigl[&\|d\phi\|^{q(p-1)}+\|\tau_p\|^q
+\|d\phi\|^{p-1}\|\nabla W\|\\
&+\|\nabla(\|d\phi\|^{p-2}d\phi)\|\,\|\tau_p\|^{q-1}\Bigr]\ddv<\infty.
\end{aligned}
\end{equation}
It implies \eqref{I} by Cauchy--Schwarz, and it makes $S_{2,p,q}$ integrable. Both $(A_q)+(B_q)$ and \eqref{strong} therefore imply \eqref{I}. The proof of Corollary \ref{global-lemma} uses \eqref{I} directly.

\begin{corollary}[A criterion for $4$-harmonicity]\label{four-cor}

Let $(M^m,g,f)$ be a complete gradient Ricci soliton with bounded Ricci curvature. Let $\phi$ be a smooth $4$-biharmonic map, namely a local critical point of $E_{4,2}=\frac12\int_M\|\tau_4(\phi)\|^2\,dv_g$, and assume

\begin{equation}\label{four-I}
\begin{aligned}
&\int_M\bigl(\|d\phi\|^6+\|\tau_4(\phi)\|^2\bigr)\,dv_g\\
&\quad+\int_M\|\nabla(\|d\phi\|^2d\phi)\|\,\|\tau_4(\phi)\|\,dv_g\\
&\quad+\int_M\|d\phi\|^2|\langle d\phi,\nabla\tau_4(\phi)\rangle|\,dv_g<\infty.
\end{aligned}
\end{equation}

Then

\begin{equation}\label{four-id}
\begin{aligned}
0={}&\int_M[\lambda(m-8)-\operatorname{Scal}^M]\|\tau_4(\phi)\|^2\,dv_g\\
&+4\int_M\operatorname{Ric}^M(e_i,e_j)
\langle\nabla(\|d\phi\|^2d\phi)(e_i,e_j),\tau_4(\phi)\rangle\,dv_g\\
&-4\int_M\langle d\phi,\nabla\tau_4(\phi)\rangle\operatorname{Ric}^M(e_i,e_j)
\langle d\phi(e_i),d\phi(e_j)\rangle\,dv_g.
\end{aligned}
\end{equation}

If $\operatorname{Scal}^M<\lambda(m-8)$ and the sum of the last two terms of \eqref{four-id} is nonnegative, then $\tau_4(\phi)=0$. The same conclusion holds when $m=2$ and scalar curvature is constant and different from $2\lambda$, without the preceding sign assumption.

\end{corollary}
\begin{proof}Set $a\equiv1$ and $p=4,q=2$ in Corollary \ref{global-lemma}, Corollary \ref{main-thm} and Corollary \ref{surface-cor}. Condition \eqref{four-I} is equivalent to the corresponding instance of \eqref{I}, since the last term differs only by a positive constant factor. The surface identity becomes

\begin{equation}\label{four-surface}
3(\operatorname{Scal}^M-2\lambda)\int_M\|\tau_4(\phi)\|^2\,dv_g=0.
\end{equation}

\end{proof}

Condition \eqref{four-I} is ensured by

\begin{equation}\label{four-AB}
\begin{aligned}
&\int_M\bigl(\|d\phi\|^6+\|d\phi\|^4\|\nabla d\phi\|^2\bigr)\,dv_g<\infty,\\
&\int_M\|d\phi\|^2|\langle d\phi,\nabla\tau_4(\phi)\rangle|\,dv_g<\infty.
\end{aligned}
\end{equation}

These are $(A_q)$ and $(B_q)$ with $p=4,q=2$. The case $p=4,q\ge2$ follows directly from Section \ref{global}, with dimension coefficient $m-4q$.

\begin{proposition}[A sufficient condition in terms of covariant derivatives]\label{four-derivative}

The following condition implies \eqref{four-I}:

\begin{equation}\label{four-deriv-cond}
\begin{aligned}
&\int_M\bigl(\|d\phi\|^4+\|d\phi\|^6
+\|d\phi\|^4\|\nabla d\phi\|^2\bigr)\,dv_g\\
&\quad+\int_M\|d\phi\|^6\|\nabla\nabla d\phi\|^2\,dv_g<\infty.
\end{aligned}
\end{equation}

\end{proposition}
\begin{proof}The product rule and trace estimates give

\begin{equation}\label{four-deriv-est}
\begin{aligned}
\|\nabla(\|d\phi\|^2d\phi)\|&\le3\|d\phi\|^2\|\nabla d\phi\|,\\
\|\tau_4(\phi)\|&\le C_m\|d\phi\|^2\|\nabla d\phi\|,\\
\|\nabla\tau_4(\phi)\|&\le C_m\bigl(\|d\phi\|\,\|\nabla d\phi\|^2
+\|d\phi\|^2\|\nabla\nabla d\phi\|\bigr).
\end{aligned}
\end{equation}

Thus the tension-square and mixed terms of \eqref{four-I} are controlled by $\|d\phi\|^4\|\nabla d\phi\|^2$. The remaining contraction satisfies

\begin{equation}\label{four-young}
\begin{aligned}
\|d\phi\|^2|\langle d\phi,\nabla\tau_4(\phi)\rangle|
&\le\|d\phi\|^3\|\nabla\tau_4(\phi)\|\\
&\le C_m\bigl(\|d\phi\|^4\|\nabla d\phi\|^2
+\|d\phi\|^5\|\nabla\nabla d\phi\|\bigr),\\
\|d\phi\|^5\|\nabla\nabla d\phi\|
&\le\tfrac12\|d\phi\|^4+\tfrac12\|d\phi\|^6\|\nabla\nabla d\phi\|^2.
\end{aligned}
\end{equation}

Integration proves the assertion. No division by $\|d\phi\|$ occurs, so the estimates remain valid at zeros of $d\phi$. Condition \eqref{four-deriv-cond} is a convenient sufficient condition, not a necessary hypothesis for the vanishing results.

\end{proof}

\subsection{Expanded boundary terms}
For homogeneous $L$ and the power density $B=\|\tau_L\|^q/q$, use $\widehat\alpha=p\alpha$ in \eqref{B-boundary}, multiply by $q$ and expand the derivative of $L_\phi'd\phi$. The right side is
\begin{equation}\label{eight-terms}
\begin{aligned}
&-2q\int_M\alpha(\Ric(\nabla\psi))\ddv
+pq\lambda\int_M\alpha(\nabla\psi)\ddv\\
&-\int_M\langle\nabla f,\nabla\psi\rangle\|\tau_L\|^q\ddv
+q\int_M(\Delta\psi)\alpha(\nabla f)\ddv\\
&+2q\int_M L_\phi'\langle(\nabla d\phi)(\nabla\psi,\nabla f),W\rangle\ddv\\
&+q\int_M\langle\nabla L_\phi',\nabla\psi\rangle
\langle d\phi(\nabla f),W\rangle\ddv\\
&+q\int_M\langle\nabla L_\phi',\nabla f\rangle
\langle d\phi(\nabla\psi),W\rangle\ddv\\
&-q\int_M\Psi_{L,q}\langle d\phi(\nabla f),d\phi(\nabla\psi)\rangle\ddv.
\end{aligned}
\end{equation}

The fifth, sixth and seventh terms must be kept together as a derivative of the product $L_\phi'd\phi$ in the estimates. If $a$ is nonconstant, the sixth, seventh and eighth terms need not vanish at $p=2$.

For $q\ge2$, away from $\tau_L=0$,
\begin{equation}\label{W-gradient}
\begin{aligned}
\nabla_XW={}&\|\tau_L\|^{q-2}\nabla_X\tau_L
+(q-2)\|\tau_L\|^{q-4}\langle\nabla_X\tau_L,\tau_L\rangle\tau_L,\\
\|\nabla W\|&\le(q-1)\|\tau_L\|^{q-2}\|\nabla\tau_L\|.
\end{aligned}
\end{equation}

For $q=2$ the second term is absent; for $q>2$ the derivative extends continuously across the zero set. For $1<q<2$, differentiability at $\tau_L=0$ is not automatic. This explains the restriction $q\ge2$ in the soliton stress argument. If $p=2$, $L_\phi''=0$ and $\Psi_{L,q}$ itself contains no derivative of $W$. No necessary integrability condition is claimed here.

\section{Explicit domain dependence}\label{general-B}
Let $L=L(x,y,r)$ and $B=B(x,y,r,s)$. The metric variation \eqref{linearize-metric} and the stress formula \eqref{B-stress} remain valid with these functions evaluated along the map. Their dependence on $x$ is held fixed during metric variation.

Explicit domain dependence introduces an additional source in the conservation equation. Denote by $d_x$ differentiation in the first variable, with all other arguments held fixed. At a critical map for the corresponding $x$-dependent functional,
\begin{equation}\label{B-source}
(\divg S_{B,L})(X)
=(d_xB)(X)-\langle d\phi,\nabla W_B\rangle(d_xL')(X)
-\langle W_B,(d_x\gradN L)(X)\rangle.
\end{equation}
To see this, pull back $g$, $\phi$, and the background functions $L,B$ simultaneously by the flow of $X$. The total variation vanishes. The map variation vanishes at a critical map, while the background variation is
\[
\int_M\bigl[(d_xB)(X)
+\langle W_B,\divg((d_xL')(X)d\phi)-(d_x\gradN L)(X)\rangle\bigr]\ddv.
\]
One integration by parts gives the right side of \eqref{B-source}; integration of the metric variation gives $-\int_M(\divg S_{B,L})(X)\ddv$. This proves the source identity. Autonomy is therefore needed in the conservation arguments of the main text. If $L$ is autonomous and $B(x,y,r,s)=b(x)\widehat B(y,r,s)$, the source is $\widehat B\,db$. The Euclidean dilation identity then contains the additional term $\int\widehat B\langle x,\nabla b\rangle\,dx$, whenever the integral and boundary terms converge.

\section*{Declaration of generative AI use}
The first work \cite{r25} was done without any AI tools ,  six month later, I used ChatGpt to improve it .   OpenAI ChatGPT assisted with the derivation and checking of formulas, the literature comparison, preparation of a partial Lean verification, and editing of the manuscript.

\end{document}